\documentclass[a4paper,10pt]{article}
\usepackage[T1]{fontenc}
\usepackage[utf8]{inputenc}
\usepackage{amssymb,amsmath,amsthm,amsfonts,mathrsfs,latexsym, hyperref, color,caption,stmaryrd,bbm}
\usepackage{graphicx,enumerate}
\usepackage{epstopdf}

\usepackage{mathptmx}
\usepackage[scaled=.95]{helvet}
\usepackage{courier}
\usepackage{float}
\usepackage{subfig}
\usepackage{subfloat}
\usepackage{verbatim}

\usepackage{todonotes}

\newtheorem{thm}{Theorem}[section]

 \newtheorem{prop}[thm]{Proposition}
 \newtheorem{lem}[thm]{Lemma}

\newcommand{\R}{\mathbb{R}}

\newcommand{\N}{\mathbb{N}}

\newcommand{\ra}{\rightarrow}
\newcommand{\vp}{\varphi}
\newcommand{\e}{\varepsilon}
\newcommand{\om}{\omega}

\begin{document}

\title{Sparse and Switching Infinite Horizon Optimal Controls with Mixed-Norm Penalizations}
\author{
Dante Kalise\thanks{School of Mathematical Sciences, University of Nottingham, University Park, Nottingham NG7 2QL, United Kingdom. Email: dante.kalise@nottingham.ac.uk . Research supported by the Imperial College London Research Fellowship program.} \and
Karl Kunisch\thanks{Institute for Mathematics and Scientific Computing, University of Graz, Heinrichstra{\ss}e 36, A-8010 Graz, Austria, and Radon Institute of Computational and Applied Mathematics, Austrian Academy of Sciences. Email: karl.kunisch@uni-graz.at. Research in part supported by the ERC advanced grant 668998 (OCLOC) under the EU's H2020 research program. } \and
Zhiping Rao\thanks{School of Mathematics and Statistics, and Hubei Key Laboratory of Computational Sciences, Wuhan University, Wuhan, 430072, Hubei, People's Republic of China. Email: zhiping.rao@whu.edu.cn. Research in part supported by NSFC 11801422.}
}
\date{}

\maketitle

\begin{abstract} A class of infinite horizon optimal control problems involving mixed quasi-norms of
 $L^p$-type cost functionals for the controls is discussed.  These functionals enhance sparsity and switching properties of the optimal controls.  The existence of optimal controls and their structural properties are analyzed on the basis of first order optimality conditions.   A dynamic programming approach is used for numerical realization.
\end{abstract}

\medskip


\section{Introduction}

 In this work we continue our  investigations of infinite horizon optimal control problems with nonconvex cost functionals which we started in \cite{KKR17}. We focus on optimal control of nonlinear  dynamical systems which are affine in the control. The input control is a vector-valued function $u=(u_1,\ldots,u_m)$ in the space $L^\infty(0,\infty;\R^m)$ under control constraints. The focus rests on that  part of the cost functional which involves  the control. It is given as follows:
\begin{equation}\label{CFpq}
\int^\infty_0 \left(\sum^m_{i=1}|u_i(t)|^p\right)^{q/p}dt,
\end{equation}
where $0<p<1$ and $p\leq q \leq 1$. This functional is nonsmooth and nonconvex, leading to a challenging optimal control problem with interesting properties for the  optimal control laws, in particular sparsity and switching. It appears that the terminology "sparse" is not rigorously defined in the literature, but generally it is used to describe the property of the optimal control to be identically  zero over nontrivial subsets of the temporal domain.  Here, by sparsity we refer to the situation in which the whole vector $u(t)$ is zero.  Switching control, is related to coordinate-wise sparsity, and is used to  describe  the property
\[
u_i(t)u_j(t)=0 \ \text{for }\ i,j\in\{1,\ldots,m\},\ i\neq j,\ t\geq 0,
\]
which is equivalent to saying that at most one coordinate of $u(t)$ is non-zero at $t$.
While the use of the control penalty  \eqref{CFpq} does not guarantee sparsity nor switching properties, it enhances them. This is illustrated in Figure \ref{uballs}, where unit balls for different $q/p$ ratios are shown. For a fixed $q$ decreasing $p$ (column-wise in the sub-figure) one direction becomes dominant over the other.

\begin{figure}[!h]
\includegraphics[width=0.32\textwidth]{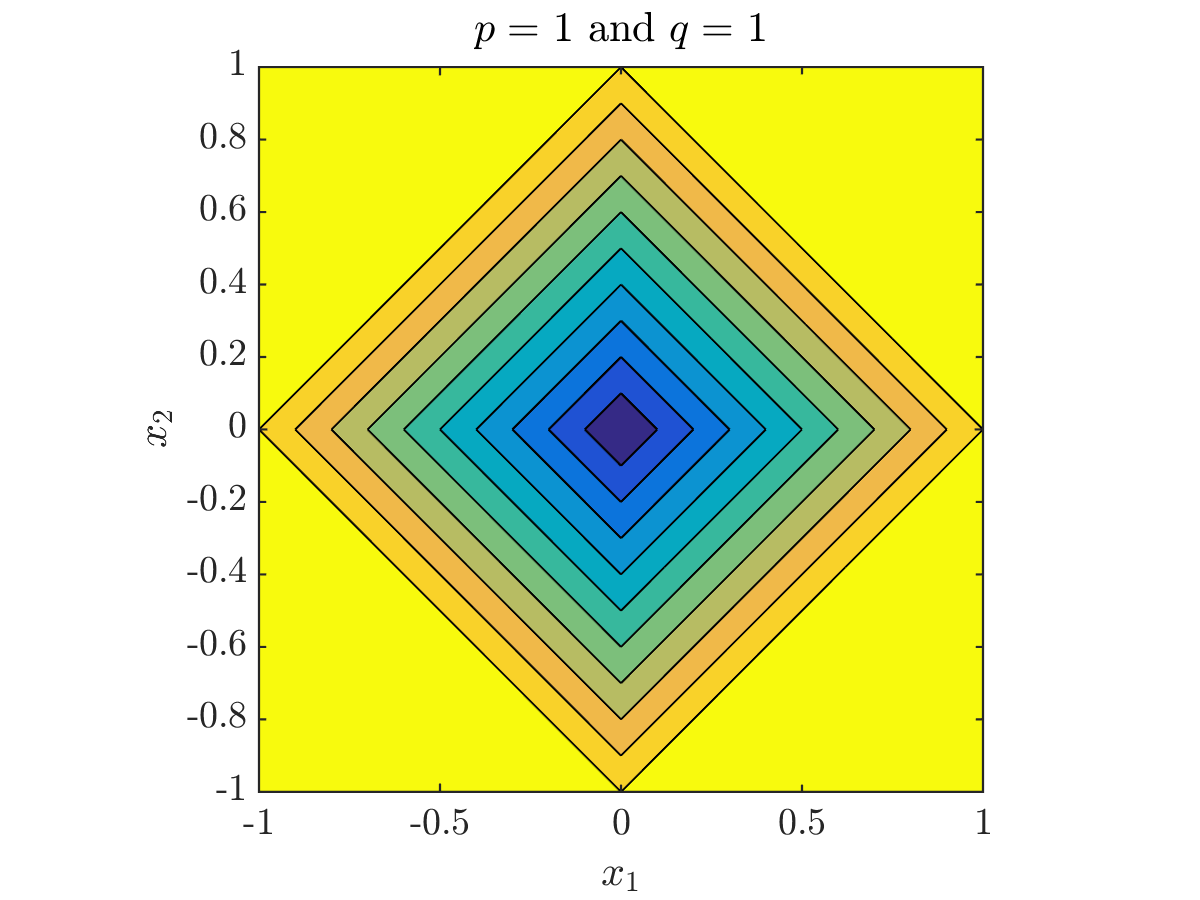}\hfill
\includegraphics[width=0.32\textwidth]{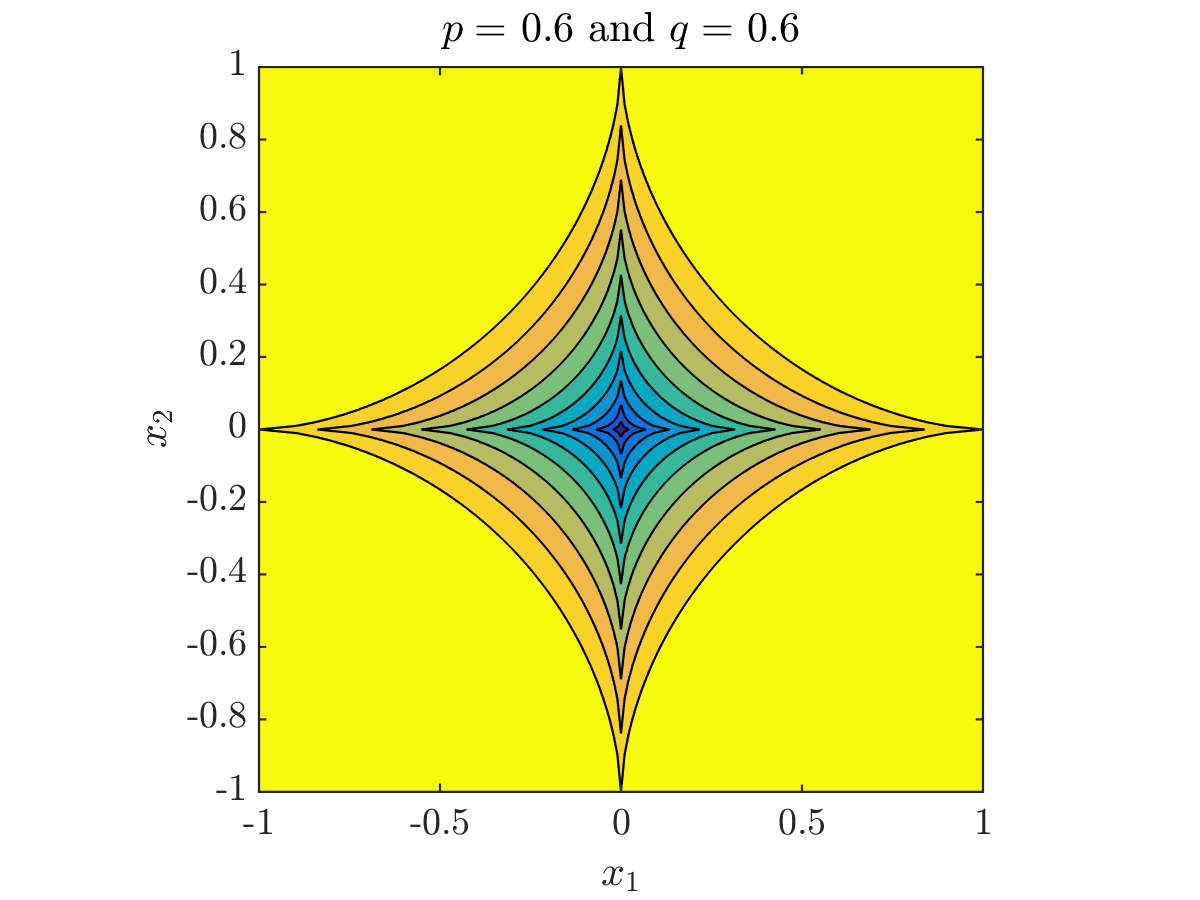}\hfill
\includegraphics[width=0.32\textwidth]{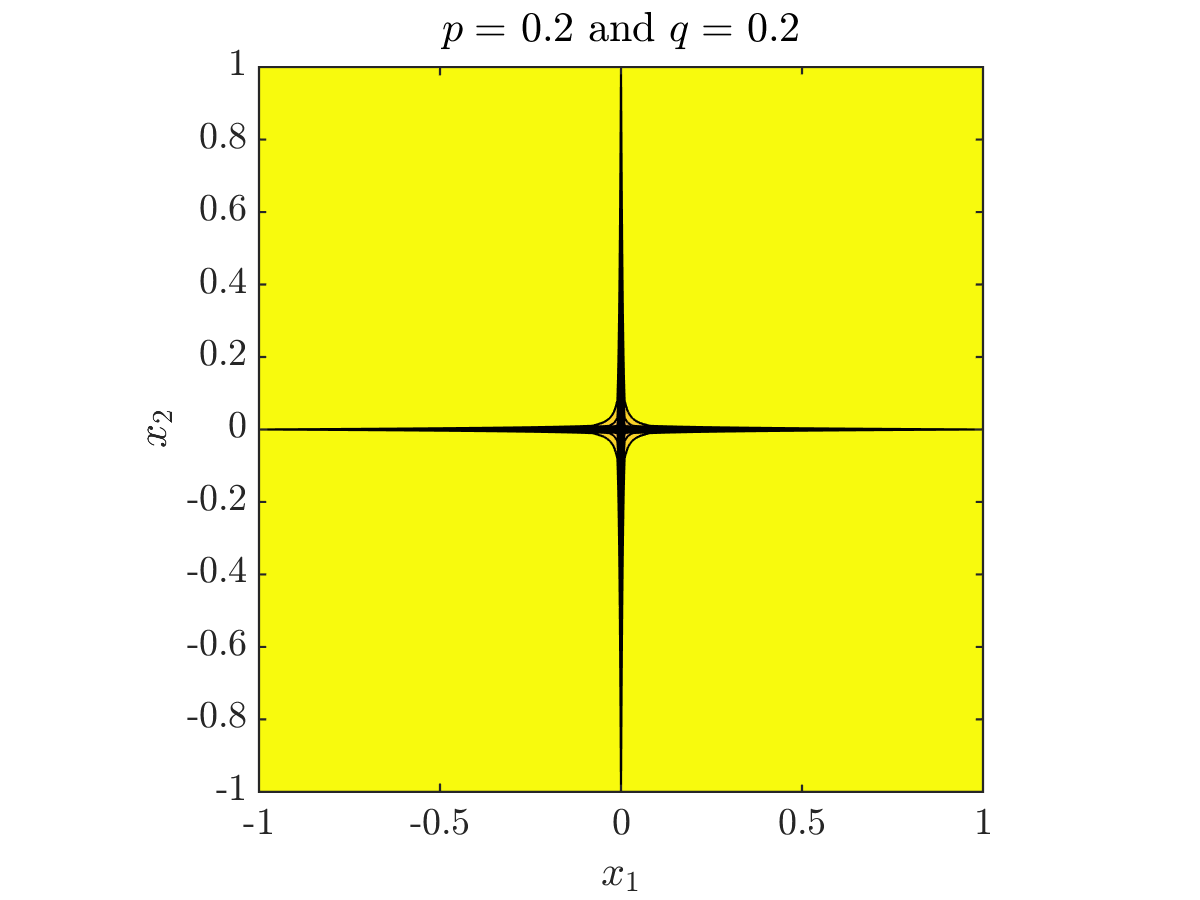}\\
\includegraphics[width=0.32\textwidth]{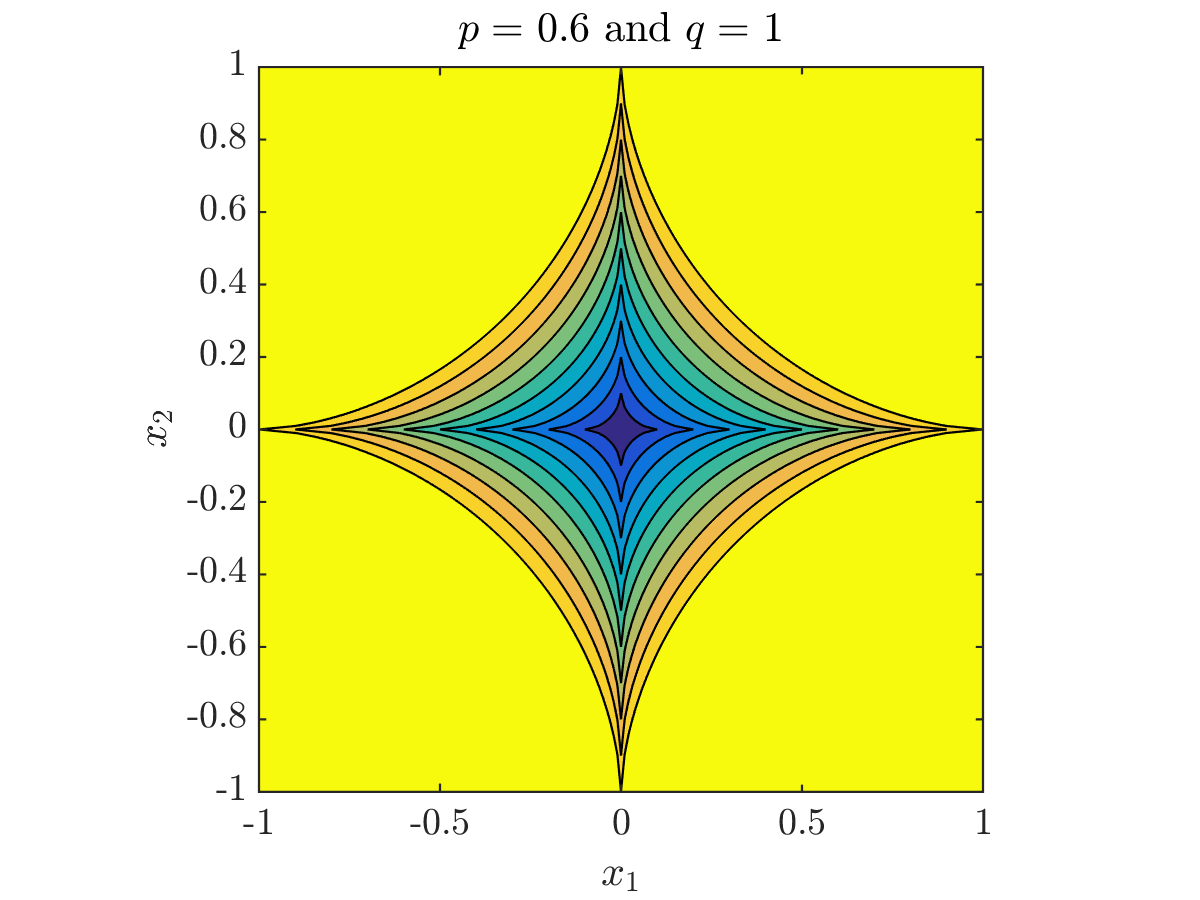}\hfill
\includegraphics[width=0.32\textwidth]{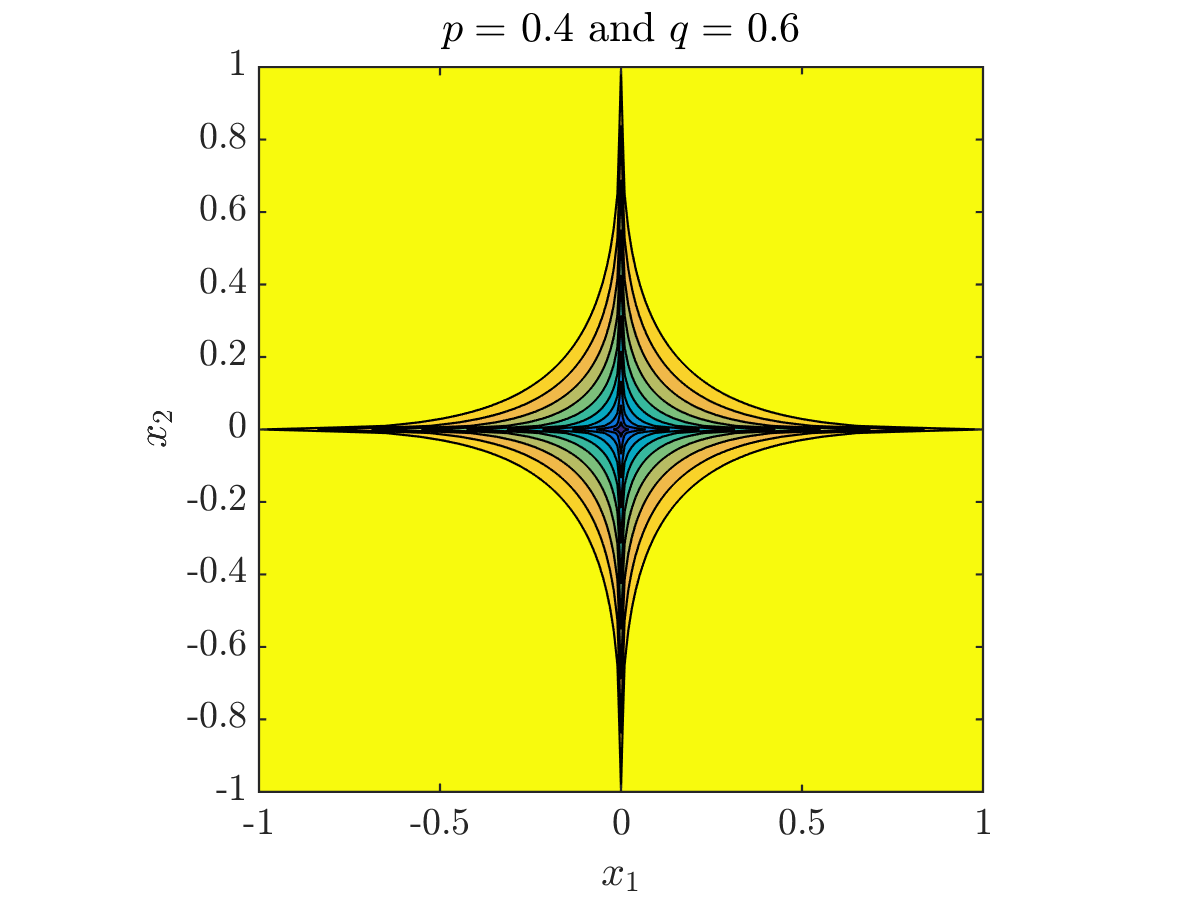}\hfill
\includegraphics[width=0.32\textwidth]{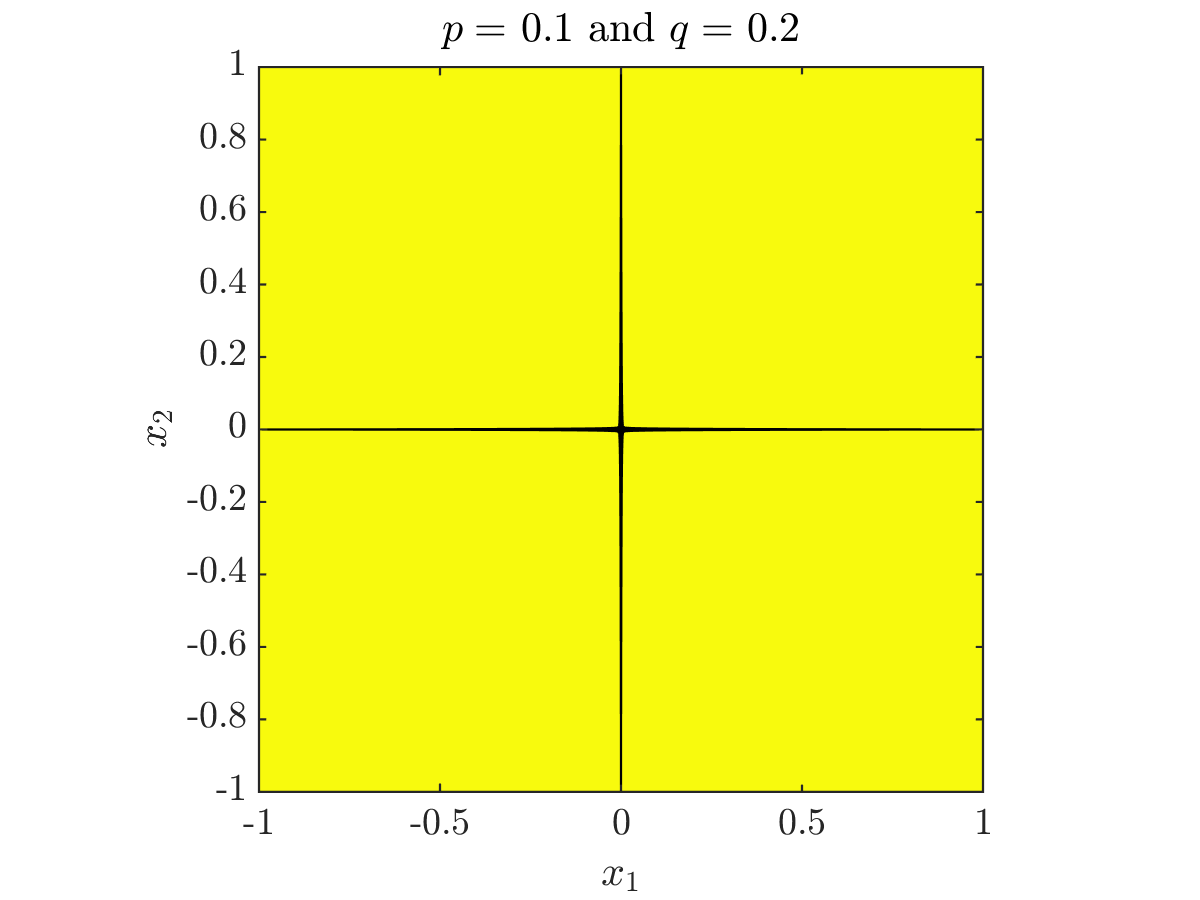}
	\caption{Contour levels ($0.1$ to $1$) of different balls $\|u\|_p^q$.}\label{uballs}
\end{figure}

To further illustrate the effect of \eqref{CFpq}  let us consider
the case $p=1/2$ and $q=1$. Then the  running cost for the control is given by
\[
\sum^m_{i=1} |u_i(t)|+2\sum_{i,j\in\{1,\ldots,m\},i\neq j}|u_i(t)u_j(t)|^{1/2},
\]
where the $L^1$-penalization on $u_i$ will support  sparsity in the control and the product penalization enhances switching phenomena. More generally, if $\frac{q}{p}=j \in \mathbb{N}$ is an integer, then
the running cost is is  combination of an $L^q$-penalization on each control coordinate $u_i$, and it further contains weighted summands of (up to) $j-$ tuples of fractional  powers of $|u_i|$, with the sum of the powers for each tuple summing to $q$. Fixing $q$, and decreasing $p$ we expect that the control cost \eqref{CFpq} increases the switching nature of the optimal  controls, since the weights on the tuples  compared to those on the singletons increase. Moreover, decreasing $q$ we expect that the subdomain over which the optimal  control vanishes (in all coordinates) increases.  These properties  will be illustrated by numerical experiments.

The case with $p=q$ and $0<p\leq 1$ has been studied in \cite{KKR17}.  Existence and  sparsity properties of optimal controls have been analyzed for this case, and these  properties have been observed in the numerical simulations in the case with $0<p=q<1$. In the present work, the analysis is made for  more general nonconvex problems with the control cost \eqref{CFpq}. Concerning the question of existence of optimal solutions, which is not guaranteed in general, we follow the ideas from  \cite{KKR17} to reformulate the problem in infinite-dimensional sequence spaces by descretizing the controls,  and  extending an important result on weakly sequentially continuous mappings from \cite{IK14} to obtain the existence result for our purposes.

The analysis of the sparsity and switching structure is based  on optimality conditions. For this purpose we derive the  necessary first order optimality  conditions of the original problem, which follow from general results which are  available  in the literature. We also derive sufficient optimality condition for the reformulated problems. Subsequently, we investigate the sparsity and switching properties of the optimal controls  under box constraints. Finally, by using dynamic programming techniques, optimal control laws are approximated globally in the state space for linear and nonlinear dynamical systems.

Let us mention previous related work on sparse  and switching control. Closed-loop infinite horizon sparse optimal control problems with $L^p$ ($0<p\leq 1$) functionals were analyzed in \cite{KKR17}. Open-loop, finite horizon $L^1$ sparse optimal control for dynamical systems have been studied in e.g. \cite{HAJ79,VM06,ALT15}. Open-loop, finite horizon sparse optimal control for partial differential equations was studied in e.g.  \cite{HSW12,CCK12,PV13}. The Hamilton-Jacobi-Bellman equation for impulse and switching controls was discussed in \cite{BC97,Y89}. The synthesis of sparse feedback laws via dynamic programming has been studied in \cite{Falcone14,KKK16,Albi17}.  In the context of partial differential equations optimal control of systems switching among different modes were analysed in \cite{HLS09,HS13}, problems with convex switching enhancing functionals  were investigated in \cite{CRKB16,CIK16}, and problems with nonconvex switching penalization in \cite{CKR17}.  In \cite{Z11}  switching controls based on functionals suggested
by controllability considerations were investigated.
 Mixed (quasi-)norms as
in  \eqref{CFpq} with $p \neq q$  have been used earlier, though typically in convex situations with
$ p\ge1, q\ge 1$. These investigations were carried out in the context of machine learning, regression analysis, and mathematical imaging, with
the goal of achieving group sparsity or structured parsimony, see e.g. \cite{BC97, Fornasier09, Kowa, WNF09, ZRY09},
and the references given there.

The structure of the paper is the following. The short section 2 contains the precise problem formulation. Existence of optimal controls, which are discretized in time, is obtained in section 3. The sparsity and switching structure  of the optimal controls is analyzed  on the basis of the optimality conditions for the  time-continuous as well as the time discrete problems in sections 4 and 5, respectively, and section 6 contains numerical results.

\section{Optimal control problem}
Let $U\subset\R^m$ be a closed set and let $f_i:\R^d\ra\R^d$ be continuous differentiable functions for $i=0,\ldots,m$. We consider the following control system: given $x\in\R^d$,
\begin{equation}\label{odey}
\left\{
\begin{array}{ll}
\dot y(t)=f_0(y(t))+\sum^m_{i=1}f_i(y(t))u_i(t) & \text{in}\ ]0,\infty[,\\
y(0)=x.
\end{array}
\right.
\end{equation}
Here $y(t)\in\R^d$ is the state variable and $u(t)=(u_1(t),\ldots,u_m(t))\in\R^m$ is the input control.
Given $p\in]0,1[$, we set for the vector $u=(u_1,\ldots,u_m)\in\R^m$
\[
\|u\|_p=\left(\sum^m_{i=1}|u_i|^p\right)^{1/p}.
\]
Let $q\in[p,1]$, $\lambda>0$, $\gamma>0$ and $y_d\in\R^d$. For any $x\in\R^d$, consider the cost functional
\begin{equation}\label{costfunctional}
J(x,u):=\int^{\infty}_0 e^{-\lambda t}\left(\frac{1}{2}\|y(t)-y_d\|^2_2+\gamma\|u(t)\|_p^q\right)dt,
\end{equation}
where $(y,u)$ satisfies the state equation \eqref{odey}, and the infinite horizon optimal control problem
\begin{equation}\label{ocpb}
\inf\left\{ J(x,u)\,:\,u\in L^{\infty}(0,\infty;U)\right\}.
\end{equation}
In \eqref{costfunctional}, $\lambda$ is called the discount factor, $\gamma$ is the weight of control cost and $\|\cdot\|_2$ is the Euclidean norm in $\R^d$. The following assumptions are made.
\begin{enumerate}
\item[(H1)] The control set $U$ is compact and convex.
\item[(H2)] There exists $L>0$ such that $\|f_i(x_1)-f_i(x_2)\|_2\leq  L \|x_1-x_2\|_2$ for all $x_1,x_2\in\R^d$, and $i=0,\ldots,m$.
\item[(H3)] For each $x\in\R^d$, there exists $u\in L^{\infty}(0,\infty;U)$ such that $J(x,u)<\infty$.
\end{enumerate}
Let us mention that the cost functional $J$ is convex in the state variable and nonconvex in the control. The case $q=p$ has been discussed in \cite{KKR17}.

\section{Time-discretized model}
Since the cost functional $J$ is not convex in $u$,
existence of optimal controllers for problem \eqref{ocpb} does not hold in general. For this purpose we  analyze  the existence in the case of a  time-discretized  approximation to  \eqref{ocpb}.
We introduce the temporal grid $(t_k)_{k\in\N}$:
\[
0=t_0<t_1<\cdots<t_k<t_{k+1}<\cdots,
\]
and denote by $I_k=[t_k,t_{k+1}[$ for $k\in\N$. The  control is then  restricted to  the following set of piecewise constant functions:
\[
U^\Delta=\{ u=(u_1,\ldots,u_m)\in L^{\infty}(0,\infty;U)\,:\,u_i(t)=u_{i,k}\ \text{for}\ t\in I_k,\ i=1,\ldots,m,\ k\in\N\}.
\]
Consider the following optimal control problem
\begin{equation}\label{ocpdtd}
\inf_{u\in U^\Delta} J^\Delta(x,u):=\int^\infty_0 e^{-\lambda t}\left(\frac{1}{2}\|y(t)-y_d\|^2_2+\gamma
\left(\sum^m_{i=1}\sum^\infty_{k=0}|u_{i,k}|^p\mathbbm{1}_{I_k}(t)\right)^{q/p}\right)dt,
\end{equation}
where $y$ solves \eqref{odey}. A direct computation shows that
\[
J^\Delta (x,u)=\int^\infty_0 e^{-\lambda t}\frac{1}{2}\|y(t)-y_d\|^2_2 dt+ \gamma\sum^\infty_{k=0}b_k\left(\sum^m_{i=1}|u_{i,k}|^p\right)^{q/p},
\]
where
\[
b_k=\int_{I_k} e^{-\lambda t}dt=\frac{1}{\lambda}(e^{-\lambda t_k}-e^{-\lambda t_{k+1}}).
\]
For any $r>0$, the infinite dimensional sequence space $\ell^r=\{u\in\ell^\infty\,:\sum^\infty_{k=1}|u_k|^r<\infty\}$ is endowed with
\[
\|u\|_r=\left(\sum^{\infty}_{k=0}|u_k|^r\right)^{1/r}.
\]
For convenience we recall that $\ell^r$, with $1<r<\infty$, are reflexive Banach spaces and $\ell^{r_1}\subset\ell^{r_2}$ if $1\leq r_1<r_2\leq\infty$.
To investigate the existence of optimal controls, we follow the idea introduced in \cite{IK14} by defining the following reparametrization. For $\{b_k\}_{k=1}^\infty$ as above we define $\ell_b^q(\ell^p)$ as the space of all $\R^m$-valued sequences $u$ such that $\sum_{k=0}^{\infty} b_k(\sum_{i=1}^{m}|u_{i,k}|^p)^{q/p} < \infty$. In particular $u_{i,k}$ denotes the i-th coordinate of the k-th sequence element $u_{\cdot,k}$. Similarly $\ell^{q/p}(\ell^1)$ stands for the space of all $\R^m$-valued sequences $w$ such that $\sum_{k=0}^{\infty} ({\color{black}\sum^m_{i=1}} |w_{i,k}|)^{q/p} < \infty$. We introduce the mapping $\Phi: \ell^{q/p}(\ell^1) \to \ell_b^q(\ell^p)$ by
\begin{equation*}
 (\Phi(w))_{i,k} = \phi(w_{i,k}), \text{ where }\phi(w_{i,k}) = b_k^{-1/q}|w_{i,k}|^{1/p} \mathop{sgn}(w_{i,k}).
\end{equation*}
{
 We have the following lemma:
\begin{lem}\label{l:KK1}
The mapping $\Phi:\ell^{q/p}(\ell^1) \to l_b^q(\ell^P)$ is an isomorphism.	
\end{lem}
\begin{proof}
For $w \in \ell^{q/p}(\ell^1)$ we have
\begin{equation}\label{eq:KK1}
\begin{aligned}
\|\Phi(w)\|_{\ell^q_b(\ell^p)} & =\left( \sum_{k=0}^{\infty} b_k \left(\sum_{i=1}^{m}|\phi(w_{i,k})|^p \right)^{q/p} \right)^{1/q}= \left(\sum_{k=0}^{\infty} b_k \left(\sum_{i=1}^{m} b_k^{-p/q}|w_{i,k}| \right)^{q/p} \right)^{1/q} \\  & = \left(\sum_{k=0}^{\infty} \left(\sum_{i=1}^{m}|w_{i,k}| \right)^{q/p} \right)^{1/q} = \|w\|_{\ell^{q/p}(\ell')}^{1/p} < \infty,
\end{aligned}
\end{equation}
and thus $\Phi(w) \in \ell^q_b(\ell^p)$ as desired.
The inverse to $\Phi$ is given by
\begin{equation*}
\left(\Phi^{-1}(u) \right)_{i,k} = b_k^{p/q} |u_{i,k}|^p \mathop{sgn}(u_{i,k}),
\end{equation*}
and this concludes the proof.
\end{proof}

From Lemma \ref{l:KK1} and \eqref{eq:KK1} it follows that \eqref{ocpdtd} is equivalent to
\begin{equation}\label{eq:KK2}
\inf_{w:(\Phi(w))_{\cdot,k}\in U} \frac{1}{2} \int_{0}^{\infty} e^{-\lambda t} \|y(t) - y_d\|^2_2 dt + \gamma \sum_{k=0}^{\infty} \left(\sum_{i=1}^{m}|w_{i,k}| \right)^{q/p},
\end{equation}
where $y(\cdot)$ satisfies
\begin{equation} \label{eq:KK3}
\begin{cases}
\begin{array}{ll}
 \dot y(t) = f_0(y(t))+\sum^m_{i=1}f_i(y(t))b_k^{-1/q}|w_{i,k}|^{1/p}\mathop{sgn}(w_{i,k}) \text{ for }\ t\in[t_k,t_{k+1}), k=0,1,\ldots,\\
 y(0) = x.
\end{array}
\end{cases}
\end{equation}

Thus the relationship between the controls on $[t_k,t_{k +1})$ in terms of $u$- and $w$- coordinates is given by $u_{i,k} = \phi(w_{i,k})$. To argue existence of \eqref{ocpdtd} or equivalently \eqref{eq:KK1} we need the following lemma whose proof is inspired by \cite[Lemma 2.1]{IK14}. We introduce the mapping
\begin{equation*}
\psi: \ell^{q/p} \to \ell^q, \text{ by } \left( \psi(z) \right)_k = |z_k|^{1/p} \mathop{sgn} z_k, \text{ for } k = 1, \ldots
\end{equation*}
for scalar valued sequences $z$.

\begin{lem}\label{l:KK2}
Let $q>p$, and let $\beta$ denote the conjugate of $q/p$. Then the mapping $\psi: \ell^{q/p} \to l^{\beta}$ is weakly (sequentially) continuous, i.e. $z^n \to \bar{z}$ weakly in $\ell^{q/p}$ implies that $\psi(z^n) \to \psi(\bar{z})$ weakly in $\ell^{\beta}$.
\end{lem}

\begin{proof}
First note that $\ell^q \subset	\ell^{\beta}$, since $\beta = \frac{q}{q-p} > q$. Let $r=\frac{1}{p}+1$ and let $r^*$ denote the conjugate exponent of $r$ given by $r^*=p+1$. Then
\[
1<\frac{q}{p}\leq \frac{1}{p}<r,
\]
which implies $r^*<\beta$. For any $z\in\ell^{q/p}$, we have
\[
\|z\|^r_r=\sum^\infty_{k=1}|z_k|^r,\ \|\psi(z)\|^{r^*}_{r^*}=\sum^\infty_{k=1}|z_k|^{r^*/p}=\sum^\infty_{k=1}|z_k|^r,
\]
and
\[
(\psi(z),z)_{\ell^{r^*},\ell^r}=\sum^{\infty}_{k=1} \psi(z)_k\cdot z_k=\sum^\infty_{k=1}|z_k|^{1/p+1}=\sum^\infty_{k=1}|z_k|^r.
\]
The above computations imply that
\[
(\psi(z),z)_{\ell^{r^*},\ell^r}=\|\psi(z)\|_{r^*}\|z\|_r, \text{ and } \|\psi(z)\|^{r^*}_{r^*}=\|z\|^r_r,
\]
which means that $\psi$ is the duality mapping from $\ell^r$ to $\ell^{r^*}$ and is weakly sequentially continuous. If $z^n\ra\bar z$ weakly in $\ell^{q/p}$, then $z^n\ra\bar z$ weakly in $\ell^r$ since $1<q/p<r$. Therefore, $\psi(z^n)\ra\psi(\bar z)$ weakly in $\ell^{r^*}$. Using that $r^*<\beta$, this implies that $\psi(z^n)\ra\psi(\bar z)$ weakly in $\ell^{\beta}$.
\end{proof}

\begin{thm}\label{thmexistence}
There exists a minimizer $\bar w$ to \eqref{eq:KK1}, and hence a minimizer $\bar u$ to \eqref{ocpdtd}.
\end{thm}
\begin{proof}
The case $p=q$ has been dealt with in \cite{KKR17}. Therefore, we focus on the case $q> p$.

Let $\left\{w^n=(w^n_1,\ldots,w^n_m)\right\}_{n=1}^\infty \subset \ell^{q/p}(\ell^1)$ denote a minimizing sequence for \eqref{eq:KK1}. We set \\ $\left\{u^n=(u^n_1,\ldots, u^n_m)\right\}_{n=1}^\infty \subset \ell^{q/p}_b(\ell^p)$, where
\begin{equation*}
u^n_{i,k} =\psi(w^n_{i,k}) = b_k^{-1/q} \left(\psi(w_i^n)\right)_k = b_k^{-1/q}|w_{i,k}^n|^{1/p} \mathop{sgn}(w_{i,k}^n), \text{ for } i=1,\ldots,m,\ k=1,\ldots,\ n=1,\ldots.
\end{equation*}
Since $\{w^n\}$ is a minimizing sequence there exists a constant $K>0$ independent of $n$, such that
\begin{equation*}
\sum^{\infty}_{k=0}|w^n_{i,k}|^{q/p}\leq \sum^{\infty}_{k=0}\left(\sum^m_{i=1}|w^n_{i,k}|\right)^{q/p} \leq K
\end{equation*}
for each $i=1,\ldots,m$. This implies that the scalar-valued sequences $\{w_{i,k}^n\}^\infty_{k=1}$ are bounded in $\ell^{q/p}$ uniformly with respect to $i=1,\ldots,m$ and $n=1,\ldots$. Hence there exists, for each $i$, a subsequence (denoted by the same symbols) and some $\bar{w}_i \in \ell^{q/p}$, such that $w_i^n \to \bar{w}_i$ in $\ell^{q/p}$, see e.g. \cite[pp. 73]{Ci90}. From Lemma \ref{l:KK2} we have that $\psi(w_i^n) \to (\bar{w}_i)$ weakly in $\ell^\beta$.

Let $y_n$ be the solution to \eqref{eq:KK2} with control $w^n$. Then on each interval $I_k$, we can deduce by  the Arzel\`a-Ascoli theorem that there exists $\bar y_k:I_k\ra\R^d$ such that
\[
y_n\ra\bar y_k \ \text{uniformly w.r.t.}\ t \in I_k,\ \text{as}\ n\ra\infty.
\]
For $\bar y:[0,\infty)\ra\R^d$ defined by $\bar y|_{I_k}=\bar y_k$ for $k\in\N$, it follows that for any $T>0$
\[
 y_n\ra \bar y\ \text{uniformly in}\ [0,T),\ \text{as}\ n\ra\infty.
\]
Therefore, $\bar y$ is the solution to \eqref{eq:KK2} corresponding to $\bar w:=(\bar w_1,\ldots,\bar w_m)$. Here we use that  the dynamics $f$ is affine in $\psi(w_i)$, $i=1,\ldots,m$. Using the fact that $y_n\ra\bar y$ pointwise in $[0,\infty)$ and $w^n_{i,k}\ra \bar w_{i,k}$ for any $i=1,\ldots,m$, $k\in\N$, we obtain by Fatou's lemma that
\begin{eqnarray*}
&& \int^\infty_0 e^{-\lambda t}\frac{1}{2}\|\bar y(t)-y_d\|^2_2 dt +\sum^{\infty}_{k=0}\left(\sum^m_{i=1}|\bar w_{i,k}|\right)^{q/p} \\
&\leq& \mathop{\lim\,\inf}_{n\ra\infty}\int^\infty_0 e^{-\lambda t}\frac{1}{2}\|y_n(t)-y_d\|^2_2dt +\sum^{\infty}_{k=0}\left(\sum^m_{i=1}|w^n_{i,k}|\right)^{q/p},
\end{eqnarray*}
which implies that $\bar w$ is a minimizer for problem \eqref{eq:KK1}, once we argue its admissability. For this purpose we observe that the weak convergence of $\psi(w_i^n)$ implies the strong convergence of $\psi(w_i^n)_k \to \psi(\bar{w}_i)_k$ for each $k$. Since$\left(b_k^{-q} \psi(w_1^n)_k,\ldots,b_k^{-q}\psi(w_m^n)_k \right) \in U$ for each $k$ this implies that \\ $\left(b_k^{-q}\psi(\bar{w}_1)_k,\ldots,b_k^{-q}\psi(\bar{w}_m)_k\right) \in U$ and thus $\bar{w}$ is admissible.
\end{proof}
}

\section{Sparsity and switching properties: the time-continuous problem}
{For the time-continuous problem \eqref{ocpb}, the Pontryagin maximum principle for infinite horizon problems has been widely studied the literature, see e.g.  \cite[Theorem 3.2]{AV17, H74} and further references provided there. We have the following result.}
\begin{lem}\label{le:k1}
{Assume that the discount factor $\lambda$ is sufficiently large.} Then
for each $x\in\R^d$, if $\bar u$ is a locally optimal control for problem \eqref{ocpb} and $\bar y$ is the associated optimal trajectory, then there exists an adjoint state $\vp:[0,\infty[\ra\R^d$ such that
{
\[
-\dot \vp(t)=Df_0(\bar y(t))^*\vp(t)+\sum^m_{i=1}Df_i(\bar y(t))^*\bar u_i(t)+e^{-\lambda t}(\bar y(t)-y_d), \text{ a.e.}\ t>0,
\]
}
$\lim_{t\to\infty} \vp(t)=0$, and  for $t\in]0,\infty[$ a.e. we have
\begin{eqnarray}\label{optimality}
&&\left\langle f_0(\bar y(t))+\sum^m_{i=1}f_i(\bar y(t))\bar u_i(t),\vp(t)\right\rangle+e^{-\lambda t}\left(\frac{1}{2}\|\bar y(t)-y_d\|^2_2+\gamma\|\bar u(t)\|_p^q\right) \nonumber\\
&\leq&\left\langle f_0(\bar y(t))+\sum^m_{i=1}f_i(\bar y(t))u_i,\vp(t)\right\rangle+e^{-\lambda t}\left(\frac{1}{2}\|\bar y(t)-y_d\|^2_2+\gamma\|u\|_p^q\right)
\end{eqnarray}
for all $u\in U$.
\end{lem}
{
\begin{proof}
We sketch a proof, verifying the assumptions of \cite[Theorem 3.2, Corollary B.5]{AV17}. It will be convenient to introduce
$f(y,u)=f_0(y(t))+\sum^m_{i=1}f_i(y(t))u_i(t) $ for the right hand side of \eqref{odey}. Condition  (A1) of the cited result holds due to the structure $f$ and the assumption that the mappings $f_i$ are in  $C^1(\R^d,\R^d)$. Since by assumption $(H1)$ the set $U$ is compact the constant $M= sup_{u\in U)}\|u\|_2$ is well-defined. Consequently by (H2) we find that $\nu=L(1+mM)$ is a global Lipschitz constant for $f(\cdot,u)$, uniformly with respect to $u\in U$. By \cite[Theorem III.5.5]{BC97}, we have that $\bar y(t)\le \tilde M\sqrt{t}e^{\nu t}$ for some constant $\tilde M$ independent of $t\in[0,\infty)$. Therefore for any $\e>0$, there exists a constant $M_\e$ such that $\bar y(t)\le M_{\e}e^{(\nu+\e)t}$. Then for $\lambda > 2 \nu $ all assumptions of \cite[Corollary B.5]{AV17} are satisfied and we can conclude the existence of an adjoint state $\vp$ satisfying \eqref{optimality}.
\end{proof}

Throughout the remainder of this section we assume that $\lambda > 2 \nu $, with $\nu$ defined in the previous proof. We point out that the property
$\lim_{t\to\infty} \vp(t)=0$, will not be needed in this section.
}
Suggested by \eqref{optimality}, we shall investigate the minimizers of the following function
\[
G_t(u):=\sum^m_{i=1}\langle f_i(\bar y(t)), \vp(t)\rangle u_i+\gamma e^{-\lambda t}\|u\|^q_p,\ \forall\,u\in U,
\]
where $t\in ]0,\infty[$.
We now assume  that the set of control constraints $U$ has the form of box constraints:
\begin{equation}\label{eq:k1}
U_\infty:=\{u=(u_1,\ldots,u_m)\in\R^m\,:\, -\rho_i\leq u_i\leq \rho_i,\ i=1,\ldots,m\},
\end{equation}
where $\rho_i>0$.   In this case the optimality condition can be used to derive the following structural properties of a minimizer.

\begin{thm}\label{THMswitching}
Let $\bar u$ be an optimal control for problem \eqref{ocpb} with $U_\infty$ given in \eqref{eq:k1}, let  $\bar y$ be the associated optimal trajectory and $\vp$  the associated adjoint state. For $t\in]0,\infty[$ a.e., we define the following index sets:
\[
I^-(t)=\{i\in\{1,\ldots,m\}\,:\, |\langle f_i(\bar y(t)),\vp(t)\rangle|\rho_i^{1-q}<\gamma e^{-\lambda t}\},
\]
\[
I^0(t)=\{i\in\{1,\ldots,m\}\,:\, |\langle f_i(\bar y(t)),\vp(t)\rangle|\rho_i^{1-q}=\gamma e^{-\lambda t}\},
\]
\[
I^+(t)=\{i\in\{1,\ldots,m\}\,:\, |\langle f_i(\bar y(t)),\vp(t)\rangle|\rho_i^{1-q}>\gamma e^{-\lambda t}\},
\]
Then the following properties hold:
\begin{enumerate}[(i)]
\item
For $t\in ]0,\infty[$ a.e. and $i\in I^-(t)$,
\[
\bar u_i(t)=0.
\]
\item
For $t\in ]0,\infty[$ a.e. and $i\in I^0(t)$,

\[
\left\{
\begin{array}{lllll}
\bar u_i(t)=0, & \text{if}\ I^+(t)\neq\emptyset,\\[1.7ex]
\bar u_i(t)\in\{0,-\rho_i\mathop{sgn}(\langle f_i(\bar y(t)),\vp(t)\rangle)\}, & \\
\bar u_i(t) \bar u_{j}(t)=0,\ i,j\in I^0(t),\ i\neq j, & \text{if}\ I^+(t)=\emptyset,\ q\in[p,1[,\\[1.7ex]
\bar u_i(t)\in [0,-\rho_i\mathop{sgn}(\langle f_i(\bar y(t)),\vp(t)\rangle)], & \\
\bar u_i(t) \bar u_{j}(t)=0,\ i,j\in I^0(t),\ i\neq j, & \text{if}\ I^+(t)=\emptyset,\ q=1.
\end{array}
\right.
\]
\item
For $t\in ]0,\infty[$ a.e. and $i\in I^+(t)$, we have
\[
\bar u_i(t)\in\{0,-\rho_i\mathop{sgn}(\langle f_i(\bar y(t)),\vp(t)\rangle)\},
\]
with $\max_{i\in I^+(t)}|\bar u_i(t)| \neq 0$.
\end{enumerate}
\end{thm}

 Let us briefly comment on sparsity and switching properties which follow from Theorem \ref{THMswitching}. For the coordinates in the index set $I^-(t)$, the controllers are zero. We refer to these coordinates as the  sparse control coordinates  at the time $t$.  If $I^+(t)= \emptyset$, then $i\in I^0(t) \cup I^-(t)$ for all $i=1,\dots, m$, and hence $u$ is switching or sparse at time $t$. If $I^+(t) \neq \emptyset$ then the coordinates in $I^0(t)$ behave like those in $I^-(t)$, they are $0$.
 The coordinates of the optimal control in the index set $I^+(t)$ are not completely determined by  (iii). They are either active, or zero and thus they join the set of sparse control coordinates.
 Comparing to the case $p=q$ which was treated in \cite[Proposition 5.2]{KKR17}, the case (iii) is such that the control is necessarily active. Thus $p<q$ enhances additional sparsity compared to $p=q$. Finally, as a consequence of the box constraints,  the optimal control is of bang-off-bang type,  except for   case (ii) with  $q=1$.

\begin{proof}
We shall use that by Lemma \ref{le:k1} we know that $\bar u(t)$ minimizes $G_t$ in $U_\infty$ for a.e. $t\in (0,\infty)$. For  convenience of notations, let us set
\[
\vp_{t,i}=\langle f_i(\bar y(t)),\vp(t)\rangle,\ \gamma_t=\gamma e^{-\lambda t}.
\]
In Step 1 below we verify (i) and (ii). The claims in (iii) are proved in Step 2.

\smallskip

{\bf Step 1: proof of (i) and (ii)}.\\
At first, let us focus on the case $q>p$. Consider further the case $\vp_{t,i}\leq 0$ for $i=1,\ldots,m$.  Then $\bar u_i(t)\geq 0$ for $i=1,\ldots,m$. We introduce
\[
\Omega:=\{u\in U_\infty\,:\,0\leq u_i\leq \rho_i,\ i=1,\ldots,m\}.
\]
Let us decompose $G_t$ in $\Omega$ as follows:
\[
G_t(u)=G_1(u)+\gamma_t G_2(u),
\]
where
\[
G_1(u)=\sum^m_{i=1}\vp_{t,i}u_i+\gamma_t\sum^m_{i=1} u_i^q,\ G_2(u)=\left(\sum^m_{i=1}u_i^p\right)^{q/p}-\sum^m_{i=1}u_i^q.
\]
$G_1$ is a concave function in $\Omega$, $G_2\geq 0$, and $G_2=0$ if and only if $\sum_{i,j=1,\ldots,m,i\neq j}|u_iu_j|=0$. Here we  use that $\frac{q}{p}\in[1,\frac{1}{p}]$ and the fact that
\[
\left(\sum^m_{i=1}a_i\right)^r\geq \sum^m_{i=1}a_i^r,
\]
for each $a_i\geq 0$, $r>1$, and equality holds if and only if $a_ia_j=0$ for all $i,j=1,\ldots,m$, $i\neq j$. Then we deduce that
\[
G_t(u)\geq G_1(u),
\]
and equality holds if and only if $\sum_{i,j=1,\ldots,m,i\neq j}|u_iu_j|=0$.

If $I^0(t)=\emptyset$ and $I^+(t)=\emptyset$, i.e. $\rho^{1-q}|\vp_{t,i}|<\gamma_t$ for $i=1,\ldots,m$, we have
\[
\vp_{t,i}u_i+\gamma_t u_i^q>0\ \text{for}\ u_i\in]0,\rho_i],
\]
where $u = u(t)$. Therefore $G_1$ attains its unique minimum at $(0,\ldots,0)$ and $G_2(0,\ldots,0)=0$. Consequently $\bar u_i(t)=0$ for $i=1,\ldots,m$.

If $I^+(t)=\emptyset$, we have
\[
\vp_{t_i}u_i+\gamma_t u_i^q>0\ \text{for}\ u_i\in]0,\rho],\ i\in I^-(t),\ \text{and}\ \vp_{t,j}u_j+\gamma_t u_j^q\geq 0\ \text{for}\ u_j\in[0,\rho_j],\ j\in I^0(t).
\]
Moreover for $j\in I^0(t)$, the expression   $\vp_{t,j}u_j+\gamma_t u_j^q$ attains its minimum in $[0,\rho_j]$ at $0$ and $\rho_j$ if $q<1$, and $\vp_{t,j}u_j+\gamma_t u_j^q\equiv 0$ if $q=1$. Therefore,
\[
\bar u_i(t)=0\ \text{for}\ i\in I^-(t),\ \bar u_j(t)\in \{0,\rho_j\}\ \text{for}\ j\in I^0(t),\ q<1,\ \text{and}\ \sum_{j,j'\in I^0(t),j\neq j'}|u_j u_{j'}|=0,
\]
and
\[
\bar u_i(t)=0\ \text{for}\ i\in I^-(t),\ \bar u_j(t) \in [0,\rho_j]\ \text{for}\ j\in I^0(t),\ q=1,\ \text{and}\ \sum_{j,j'\in I^0(t),j\neq j'}|u_j u_{j'}|=0.
\]

If $I^+(t)\neq\emptyset$, we have
\[
\vp_{t,i}u_i+\gamma u_i^q\geq 0\ \text{for}\ u_i\in[0,\rho_i],\ i\in I^-(t)\cup I^0(t),
\]
and $\vp_{t,j}u_j+\gamma u_j^q$ attains its unique minimum in $[0,\rho_j]$ at $\rho_j$ for $j\in I^+(t)$. Thus, for any $u\in\Omega$, we define $\tilde u\in\Omega$ as follows:
\[
\tilde u_i=0\ \text{for}\ i\in I^-(t)\cup I^0(t),\ \text{and}\ \tilde u_i=u_i\ \text{for}\ i\in I^+(t).
\]
If $\tilde u=(0,\ldots,0)$, then for any $j\in I^+(t)$ we set $\hat u\in \Omega$ with
\[
\hat u_j=\rho\ \text{and}\ \hat u_i=0\ \text{for}\ i\neq j,\ i=1,\ldots,m.
\]
Thus  we have
\[
G(u)\geq G_1(u)\geq G_1(\tilde u)>G_1(\hat u)=G(\hat u).
\]
Otherwise if $\tilde u\neq (0,\ldots,0)$ and $u\neq\tilde u$,
\[
G(u)=G_1(u)+\gamma_t G_2(u)> G_1(u)+\gamma_t G_2(\tilde u)\geq G_1(\tilde u)+\gamma_t G_2(\tilde u)=G(\tilde u).
\]
We then deduce that
\[
\bar u_i(t)=0\ \text{for}\ i\in I^-(t)\cup I^0(t).
\]
The proof for the case when $\vp_{t,i}\leq 0$ for $i=1,\ldots,m$ is thus concluded. The other cases when $\vp_{t,i}$ have different signs can be treated analogously.

Now we proceed to look at  the case $q=p$. In this situation, $G_2\equiv 0$ and $G\equiv G_1$. The minimizers of $G_1$ have been analyzed in the previous arguments, and  we therefore arrive at the conclusion.

\smallskip

{\bf Step 2: proof of (iii)}.\\
We turn to analyze the behavior of the coordinates with indices in
$I^+(t)$. In particular in this case $I^+(t)\neq\emptyset$, and consequently by (i) and (ii)
\[
\bar u_i(t)=0,\ \text{for}\ i\in I^-(t)\cup I^0(t).
\]
Therefore,
\begin{equation}\label{defGt+}
G_t(\bar u(t))=\sum^\ell_{\tau=1}\vp_{t,i_\tau}\bar u_{i_\tau}(t)+\gamma_t\left(\sum^\ell_{\tau=1}|\bar u_{i_\tau}(t)|^p\right)^{q/p},
\end{equation}
where $\{i_1,\ldots,i_\ell\}\subset\{1,\ldots,m\}$ is such that $I^+(t)=\{i_1,\ldots,i_\ell\}$.
Then the problem consists in finding the minimizer of the function
\begin{equation}\label{EQtildeG}
\tilde G(w):=\sum^\ell_{\tau=1}\psi_\tau w_\tau + \gamma_t\left(\sum^\ell_{\tau=1}\rho_\tau^p |w_{\tau}|^p\right)^{q/p},\ \text{for}\ w=(w_1,\ldots,w_\ell)\in [-1,1]^\ell,
\end{equation}
where, to simplify notation, we set for $\tau=1,\ldots,\ell$
\begin{equation}\label{EQwpsi}
w_\tau=\frac{u_{i_\tau}}{\rho_{i_\tau}},\ \psi_\tau=\vp_{t,i_\tau}\rho_{i_\tau},\ \text{and}\ \rho_\tau=\rho_{i_\tau}.
\end{equation}
Following the definition of $I^+(t)$, we have
\begin{equation}\label{EQpsi}
|\psi_\tau|\rho_\tau^{-q}>\gamma_t,\ \text{for}\ \tau=1,\ldots,\ell.
\end{equation}
Let $\bar w$ be the minimizer and let us start by considering the case
\[
\psi_\tau<0,\ \text{for all }\ \tau=1,\ldots,\ell.
\]
Then it is trivial to see that
\[
\bar w_\tau\geq 0,\ \text{for all}\ \tau=1,\ldots,\ell.
\]
We aim to prove that the minimizer $\bar w$ is not in the interior of $[0,1]^\ell$. Without loss of generality, we assume that
\[
1\geq \bar w_1\geq \bar w_2\geq\cdots\geq \bar w_\ell\geq 0.
\]
We can therefore limit our attention to the subset
\begin{equation}\label{EQuorder}
\{(w_1,\ldots,w_\ell)\,:\, 1\geq w_1\geq w_2\geq\cdots w_\ell\geq 0\}.
\end{equation}
Note that $\bar w$ can be expressed as $\bar w= (\beta_0 \bar w_1,\beta_1\beta_0 \bar w_1,\ldots,\beta_{\ell-1}\ldots\beta_0 \bar w_1)$ where $\beta_0=1$, and $\beta_\tau\in[0,1],\ \tau=1,\dots,\ell-1$. Moreover $\bar w_1 \in [0,1]$ is a minimizer of the functional
\[
G_{\beta,1}(w_1)=\sum^\ell_{\tau=1}\psi_\tau\beta_{\tau-1}\cdots\beta_0 w_1+\gamma_t\left(\sum^\ell_{\tau=1}\rho_\tau^p\beta_{\tau-1}^p\cdots\beta_0^p\right)^{q/p}w_1^q.
\]
We will exclude the case that $ w_1 \to G_{\beta,1}(w_1)$ assumes a minimum in the interior of $[0,1]$.

Indeed, if such a  minimum $w_1^*$ is attained in the interior of $[0,1]$, then
\[
0=G_{\beta,1}'(w_1^*)=\sum^\ell_{\tau=1}\psi_\tau\beta_{\tau-1}\cdots\beta_0+\gamma_t\left(\sum^\ell_{\tau=1}\rho_\tau^p\beta_{\tau-1}^p\cdots\beta_0^p\right)^{q/p}q(w_1^*)^{q-1}.
\]
Therefore,
\[
G_{\beta,1}(w_1^*)=(1-q)\gamma_t\left(\sum^\ell_{\tau=1}\rho_\tau^p\beta_{\tau-1}^p\cdots\beta_0^p\right)^{q/p}q(w_1^*)^{q}\geq 0.
\]
Note that
\[
G_{\beta,1}(1,0,\ldots,0)=\psi_1+\gamma_t\rho_1^q=\rho_1^q(-|\psi_1|\rho_1^{-q}+\gamma_t)<0,
\]
where \eqref{EQpsi} is applied. Thus,
\[
G_{\beta,1}(w_1^*)>G_{\beta,1}(1,0,\ldots,0),
\]
which contradicts the assumption that $w_1^*$ is the minimizer.
Consequently, the minimum can not be attained in the interior of $[0,1]$ and thus $\bar w_1\in\{0,1\}$. Moreover $G_{\beta,1}(1,0,\ldots,0)<0$ and $G_{\beta,1}(0,\ldots,0)=0$, and thus

\begin{equation}\label{EQw1}
\bar w_1=1.
\end{equation}
We next  claim the following: for $j\in\{2,\ldots,\ell-1\}$, if $\bar w_{j-1}\in\{0,1\}$, then
\begin{equation}\label{EQClaim}
\bar w_{j}\in\{0,1\},
\end{equation}
and verify this statement by induction. If $\bar w_{j-1}=0$, by \eqref{EQuorder} we have
\[
\bar w_{j}=0
\]
as claimed. If $\bar w_{j-1}=1$, then
$
\bar w_\tau=1,\ \text{ for all  } \,\tau=1,\ldots,j-1.
$

To characterize further  $\bar w_j$, we apply the same idea as for determining $\bar w_1$. This time we restrict our attention to the subset
\[
\{(w_j,w_{j+1},\ldots, w_\ell)\,:\,1\geq w_j\geq w_{j+1}\geq \cdots\geq w_\ell\geq 0\},
\]
and note that for the optimal $(\bar w_j,\dots, \bar w_\ell)= (\bar w_j,\beta_j \bar w_j,\ldots, \beta_{\ell-1}\dots\beta_j \bar w_j)$,
where $\beta_\tau\in[0,1]$ for $\tau=1,\ldots,\ell-1$.  We denote for any $w_j\in[0,1]$
\begin{eqnarray*}
G_{\beta,j}(w_j)&=&\sum^{j-1}_{\tau=1}\psi_\tau+\psi_jw_j+\sum^\ell_{\tau=j+1}\psi_\tau\beta_{\tau-1}\cdots\beta_jw_j\\
&&+\gamma_t\left[\sum^{j-1}_{\tau=1}\rho_\tau^p+\rho_j^p w_j^p+\sum^\ell_{\tau=j+1}\rho_\tau^p\beta^p_{\tau-1}\cdots\beta_j^p w_j^p\right]^{q/p},
\end{eqnarray*}
and note that $\bar w_j$ is a minimizer of $G_{\beta,j}$ on $[0,1]$.
If a minimum $w_j^*$ is attained in the interior of $[0,1]$, then
\[
G_{\beta,j}'(w_j^*)=0.
\]
This yields that,
\[
\psi_j+\sum^\ell_{\tau=j+1}\psi_\tau\beta_{\tau-1}\cdots\beta_j+\gamma_t\frac{q}{p}S_j^{q/p-1}\left(\rho_j^p+\sum^\ell_{\tau=j+1}\rho_\tau^p\beta_{\tau-1}^p\cdots\beta_j^p\right)p(w_j^*)^{p-1}=0,
\]
where
\[
S_j=\sum^{j-1}_{\tau=1}\rho_\tau^p+\rho_j^p (w_j^*)^p+\sum^\ell_{\tau=j+1}\rho_\tau^p\beta^p_{\tau-1}\cdots\beta_j^p (w_j^*)^p.
\]
Therefore,
\[
\psi_j+\sum^\ell_{\tau=j+1}\psi_\tau\beta_{\tau-1}\cdots\beta_j=-\gamma_t S_j^{q/p-1}\left(\rho_j^p+\sum^\ell_{\tau=j+1}\rho_\tau^p\beta_{\tau-1}^p\cdots\beta_j^p\right)q(w_j^*)^{p-1}.
\]
By applying the above equality to compute $G_{\beta,j}(w_j^*)$, we obtain
\begin{eqnarray*}
&& G_{\beta,j}(w_j^*)\\
&=& \sum^{j-1}_{\tau=1}\psi_\tau+\left(\psi_j+\sum^\ell_{\tau=j+1}\psi_\tau\beta_{\tau-1}\cdots\beta_j\right)w_j^*+\gamma_t S_j^{q/p}\\
&=& \sum^{j-1}_{\tau=1}\psi_\tau -\gamma_t S_j^{q/p-1}\left(\rho_j^p+\sum^\ell_{\tau=j+1}\rho_\tau^p\beta_{\tau-1}^p\cdots\beta_j^p\right)q(w_j^*)^p+\gamma_t S_j^{q/p}\\
&=& \sum^{j-1}_{\tau=1}\psi_\tau+\gamma_t S_j^{q/p-1}\left[\sum^{j-1}_{\tau=1}\rho_
\tau^p + (1-q)\left(\rho_j^p+\sum^\ell_{\tau=j+1}\rho_\tau^p\beta_{\tau-1}^p\cdots\beta_j^p\right)(w_j^*)^p\right].
\end{eqnarray*}
Using the fact that $q\leq 1$ and $w_j^*>0$, it holds that
\begin{eqnarray*}
G_{\beta,j}(w_j^*)&\geq &\sum^{j-1}_{\tau=1}\psi_\tau+\gamma_t S_j^{q/p-1}\sum_{\tau=1}^{j-1}\rho_\tau^p\\
&>& \sum^{j-1}_{\tau=1}\psi_\tau+\gamma_t \left(\sum^{j-1}_{\tau=1}\rho_\tau^p\right)^{q/p-1}\sum_{\tau=1}^{j-1}\rho_\tau^p\\
&=& \sum^{j-1}_{\tau=1}\psi_\tau+\gamma_t \left(\sum^{j-1}_{\tau=1}\rho_\tau^p\right)^{q/p}\\
&=& G_{\beta,j}(0),
\end{eqnarray*}
which contradicts the assumption that $w_j^*$ is the minimizer.
Consequently, the minimum can not be attained in the interior of $[0,1]$. We then deduce that
\[
\bar w_j\in\{0,1\},
\]
which completes the proof for the claim \eqref{EQClaim}. Together with \eqref{EQw1}, it is deduced that
\[
\bar w_\tau\in\{0,1\},\ \text{for}\ \tau=1,\ldots,\ell,
\]
which concludes the case where $\psi_\tau<0$ for all $\tau=1,\ldots,\ell$.

For the other cases where $\psi_\tau$ is positive for some $\tau\in \{1,\ldots,\ell\}$, $\psi_\tau$ and $w_\tau$ can be replaced by $-\psi_\tau$ and $-w_\tau$  in \eqref{EQtildeG}. Then by following the same arguments as in the previously we can obtain that $-\bar w_\tau\in \{0,1\}$. Therefore we conclude that
\[
\bar w_\tau\in\{0,-\mathop{sgn}(\psi_\tau)\},\ \text{for}\ \tau=1,\ldots,\ell,
\]
with the additional information that $|\bar w_1|=1$.
The definition of $w_\tau$ and $\psi_\tau$ in \eqref{EQwpsi} implies that
\[
\bar u_{i_\tau}\in\{0,-\rho_{i_\tau}\mathop{sgn}(\varphi_{t,i_{\tau}})\},\ \text{for}\ i_\tau\in I^+(t),\ \tau=1,\ldots,\ell,
\]
with the additional information that   $\max_{i\in I^+} \frac{|\bar u_i|}{\rho_i} \neq 0$. This completes the proof of (iii).
\end{proof}

 In Theorem \ref{THMswitching} the  study has been made for the case of  box constraints.  Next we briefly consider the problem under Euclidean norm constraints. In this case, due to the coupling of the coordinates which is inherent to the Euclidean norm,  it appears to be more complicated
 to achieve explicit information on the structure of the minimizers compared to that which was obtained for box constraints.

 We define for $\rho>0$
\begin{equation}\label{eq:k2}
U_2:=\{u=(u_1,\ldots,u_m)\in\R^m\,:\,\sum^m_{i=1}u^2_i\leq \rho^2\}.
\end{equation}

\begin{thm}\label{THMswitching2}
Let $\bar u$ be an optimal control for problem \eqref{ocpb} with $U$ given in \eqref{eq:k2}, let  $\bar y$ be the associated optimal trajectory, and $\vp$  its  associated adjoint state.
Let $I^-(t), I^0(t)$ and $I^+(t)$ be as defined in Theorem \ref{THMswitching}. If for some $t\in ]0,T[$ the cardinality of $I^+(t)$ is less or equal to 1,
  then (i), (ii),  and (iii) of that theorem remain valid.  Otherwise we have
\begin{equation}\label{eqkk20}
\sum_{i=1}^m |\bar u_i(t)|^2 = \rho^2, \text{ for a.e. } t \in ]0,\infty[.
\end{equation}
\end{thm}
\begin{proof}

{ \em{Step 1.}} From \eqref{optimality} we know that for $t\in ]0,\infty[$ a.e., $\bar u(t)$ is the minimizer of the following function
\[
\bar G_t(u):=\sum^m_{i=1} \alpha_i(t)  u_i +\gamma e^{-\lambda t}\|u\|^q_p,\ \text{ for all } \,u\in U_2,
\]
where $\alpha_i(t)=\langle f_i(\bar y(t)),\vp(t)\rangle $.
At first we note  that $U_2$ is a subset of $U_\infty$, if $\rho_i= \rho$ for all $i$, and hence $\min_{u\in U_\infty} \bar G_t(u) \le \min_{u\in U_2} \bar G_t(u)$.  Moreover, if  a minimizer of $\bar G_t$ over $U_\infty$ is contained in $U_2$, then this minimizer is also a minimizer of $\bar G_t$ over $U_2$.
Following this observation, let
 $\bar u(t)$ be a minimizer of $\bar G_t$ over  $U_\infty$ with cardinality
 of $I^+(t) \le 1$. Then by Theorem \ref{THMswitching} all components of $\bar u(t)$ are 0 except for at most one. In case the cardinality of $I^+(t)$ equals one, then there is one non-trivial  coordinate of the control at time $t$ whose norm then equals $\rho$.

 { \em{Step 2.}} Now we turn to the general case (assuming that $I^+$ is nonempty) and prove that the optimal control is necessarily active. Since $I^+(t)$ is non-empty there exists at least one index $\tau$ such that $\gamma_t-|\alpha_\tau(t)|\rho^{1-q}<0$. Setting the value of this coordinate equal to $\rho$ we obtain
 \[
 G(( 0,\ldots0,\rho,0,\ldots,0))=\alpha_\tau(t)\rho+\gamma_t \rho^q=\rho^q(\gamma_t+\alpha_\tau(t)\rho^{1-q})=\rho^q(\gamma_t-|\alpha_\tau(t)|\rho^{1-q})<0,
 \]
 which implies that at least one coordinate of $\bar u$ is nontrivial and $G(\bar u(t)) < 0$. Let $\tilde \ell$ denote the number of nontrivial coordinates of $\bar u$ and without loss of generality assume that these are the $\tilde \ell$ first ones of $\bar{u}(t)$.

 Let us start with the case where $\alpha_i(t)\le 0$ for all $i=1,\ldots,\tilde \ell$. It is trivial to see that $\bar u_i(t)\geq 0$, for $i=1,\ldots, \tilde \ell$ in this case.  We set
\[
\Omega:=\left\{(u_1,\ldots,u_{\tilde \ell})\in\R^{\tilde \ell}\, :\, \sum^{\tilde \ell}_{i=1}u_i^2\leq\rho^2,\ u_i\geq 0,\ i=1,\ldots,\tilde \ell \right\}.
\]
Thus $\bar u(t)\in\Omega$. We prove by contradiction that $\bar u\in\partial\Omega$. If this is not the case, i.e. $\bar u$ is in the interior of $\Omega$, then
\[
\frac{\partial G}{\partial u_i}(\bar u(t))=0,\ i=1,\ldots,\tilde \ell,
\]
from which we deduce that
\[
\alpha_i(t)+\gamma_t\frac{q}{p}\left(\sum^{\tilde \ell}_{j=1}|\bar u_j(t)|^p\right)^{q/p-1} p|\bar u_i(t)|^{p-1}=0,\ i=1,\ldots,{\tilde \ell}.
\]
It follows that
\[
\alpha_i(t) \bar u_i(t)=-\gamma_t q\left(\sum^{\tilde \ell}_{j=1}|\bar u_j(t)|^p\right)^{q/p-1}|\bar u_i(t)|^{p}=0,\ i=1,\ldots,{\tilde \ell}.
\]
Therefore,
\begin{eqnarray*}
G(\bar u(t))&=& \sum^{\tilde \ell}_{i=1}\alpha_i(t) \bar u_i(t) +\gamma_t\left(\sum^{\tilde \ell}_{i=1}|\bar u_i(t)|^p\right)^{q/p}\\
&=& -\gamma_t q\left(\sum^{\tilde \ell}_{i=1}|\bar u_i(t)|^p\right)^{q/p-1}|\bar u_i(t)|^{p}+\gamma_t\left(\sum^{\tilde \ell}_{i=1}|\bar u_i(t)|^p\right)^{q/p}\\
&=& \gamma_t (1-q) \left(\sum^{\tilde \ell}_{i=1}|\bar u_i(t)|^p\right)^{q/p} \; \geq  \;0.
\end{eqnarray*}
Since we already know that $G(\bar u(t)) <0$ this gives a contradiction. Consequently $\bar u(t)\in\partial\Omega$. Since $\bar u_i \neq 0$ for $i=1,\dots,{\tilde \ell}$ this implies that
\begin{equation*}
\sum^{\tilde \ell}_{i=1}\bar u_i^2(t)=\rho^2\ \text{and}\ \bar u_i(t)> 0,  \text{ for }\ i=1,\ldots,{\tilde \ell}.
\end{equation*}
If some of the coordinates of $\alpha$ are such that $\alpha_i(t) \geq 0$, then necessarily $\bar u_i(t) \leq 0$ and, adapting $\Omega$ accordingly, it can again be verified that $\sum_{1}^{l} \bar u_i^2(t) = \rho^2$.
\end{proof}

\section{Sparsity and switching properties: the time-discretized problem}
In this subsection we consider the following linear dynamical system: for $x\in\R^d$,
\begin{equation}\label{dslinear}
\left\{
\begin{array}{ll}
 \dot y(t)=Ay(t)+Bu(t),\\
 y(0)=x,
\end{array}
\right.
\end{equation}
where $A\in\R^{d\times d}$ and $B\in\R^{d\times m}$. Let us recall the optimal control problem: given $x\in\R^d$, consider
\begin{equation}\label{OCJDelta}
\inf\left\{J^\Delta (x,u)\, :\, (y,u)\ \text{satisfies}\ \eqref{dslinear},\ u\in U^\Delta\right\}.
\end{equation}
The cost functional is recalled as follows:
\[
J^\Delta(x,u)=\int^\infty_0 \frac{1}{2}e^{-\lambda t}\|y(t)-y_d\|^2_2dt+\gamma R(u),
\]
where
\[
R(u)=\sum^\infty_{k=0}b_k\left(\sum^m_{i=1}|u_{i,k}|^p\right)^{q/p}\ \text{for}\ u\in U^\Delta.
\]
To investigate the optimality conditions satisfied by the optimal controllers, we introduce firstly the adjoint equation associated to  $(y,u)$ satisfying \eqref{dslinear}:
\begin{equation}\label{Eqadjoint}
\left\{
\begin{array}{ll}
 -\dot \vp(t)=A^T\vp(t)+e^{-\lambda t}(y(t)-y_d) & \text{for}\ t>0,\\
 \lim_{t\ra\infty} \vp(t)=0.
\end{array}
\right.
\end{equation}
{
In the remainder of this section we assume that there exists an adjoint state $\vp$ satisfying \eqref{Eqadjoint}. The following result justifies this assumption. We denote by $\sigma(A)$ the set of eigenvalues of $A$, and $\sigma:=\sup\{\mathop{Re}\,\mu\,:\, \mu\in\sigma(A)\}$. Further we define
\[
L^2_\lambda(0,\infty;\R^d):=\{y\in L^2(0,\infty;\R^d)\}\,:\, e^{\lambda \cdot}y\in L^2(0,\infty;\R^d)\},
\]
with $\|y\|^2_{L^2_\lambda(0,\infty;\R^d)}=\int^\infty_0 e^{\lambda t}|y(t)|^2dt$.

\begin{lem}
If $\lambda>2\max\{0,\sigma\}$, then \eqref{Eqadjoint} admits a unique solution $\vp\in L^2_\lambda(0,\infty;\R^d)$ with $\dot\vp\in L^2_\lambda(0,\infty;\R^d)$.
\end{lem}
\begin{proof}

If $y$ is feasible for \eqref{OCJDelta}, and in particular if it is optimal, then as a consequence of the cost functional and the inhomogeneity in \eqref{Eqadjoint} it follows that
\[
e^{-\lambda t}(y(t)-y_d)\in L^2_\lambda(0,\infty;\R^d).
\]
This suggests to investigate, for arbitrary $g\in L^2_\lambda(0,\infty;\R^d)$, the following equation has a solution:
\begin{equation}\label{adjointg}
\left\{
\begin{array}{ll}
-\dot \vp(t)=A^T\vp(t)+g(t),\ \text{a.e.}\ t>0,\\
\lim_{t\ra\infty}\vp(t)=0.
\end{array}
\right.
\end{equation}
We define
\[
W^{1,2}_\lambda:=\{y\in L^2_\lambda(0,\infty;\R^d)\,:\,\dot y\in L^2_\lambda(0,\infty;\R^d)\},
\]
endowed with $\|y\|^2_{W^{1,2}_\lambda}=\|y\|^2_{L^2_\lambda(0,\infty;\R^d)}+\|\dot y\|^2_{L^2_\lambda(0,\infty;\R^d)}$ as norm.
Note that $y\in W^{1,2}_\lambda$ implies that $t\mapsto y$ is continuous on $[0,\infty)$. Hence $W^{1,2}_{\lambda,0}=\{y\in W^{1,2}_\lambda\,:\, y(0)=0\}$ is well defined. For any $\vp\in W^{1,2}_{\lambda,0}$ we have $e^{\frac{\lambda\cdot}{2}}\vp\in L^2(0,\infty;\R^d)$ and $\frac{d}{dt}\left(e^{\frac{\lambda\cdot}{2}}\vp\right)\in L^2(0,\infty;\R^d)$. This implies that $\lim_{t\ra \infty} e^{\frac{\lambda t}{2}}\vp(t)=0$ and hence $\lim_{t\in\infty}\vp(t)=0$.

By the definition of $\sigma$ there exists a constant $M>0$ such that
\[
\|e^{At}y\|_2\leq M e^{\sigma t}\|y\|_2,\ \text{for all } t\in [0,\infty)\ \text{and}\ y\in\R^d.
\]
We now define the bounded linear operator $T:W^{1,2}_{\lambda,0}\ra L^2_{\lambda}(0,\infty;\R^d)$ by
\[
Ty=\dot y-Ay+\lambda y.
\]
The adjoint operator $T^*: L^2_\lambda(0,\infty;\R^d)\ra \left(W^{1,2}_{\lambda,0}\right)^*$ satisfies
\begin{eqnarray*}
\langle T^*\vp,y\rangle_{\left(W^{1,2}_{\lambda,0}\right)^*,W^{1,2}_{\lambda,0}}&=& \langle \vp,\dot y\rangle_{L^2_\lambda(0,\infty;\R^d)} -\langle\vp, Ay-\lambda y\rangle_{L^2_\lambda(0,\infty;\R^d)}\\
&=& -\langle \dot \vp, y\rangle_{\left(W^{1,2}_{\lambda,0}\right)^*,W^{1,2}_{\lambda,0}} -\langle\lambda\vp,y\rangle_{L^2_\lambda(0,\infty;\R^d)}-\langle A^T\vp -\lambda\vp,y\rangle_{L^2_\lambda(0,\infty;\R^d)}\\
&=& -\langle \dot \vp, y\rangle_{\left(W^{1,2}_{\lambda,0}\right)^*,W^{1,2}_{\lambda,0}}-\langle A^T\vp,y\rangle_{L^2_\lambda(0,\infty;\R^d)},
\end{eqnarray*}
where we use that $y(0)=0$ and $\lim_{t\ra\infty} e^{\frac{\lambda t}{2}}y(t)=0$ in the second equality above. Thus $T^*\vp=-\dot \vp-A^T\vp$ in the variational sense.

Let us next argue that $T$ is surjective. We choose $f\in L^2_\lambda(0,\infty;\R^d)$ arbitrarily and define
\[
y(t)=\int^t_0 e^{(A-\lambda I)(t-s)}f(s)ds.
\]
Then we have
\begin{equation}\label{dsf}
\dot y=Ay-\lambda y+f
\end{equation}
and
\begin{eqnarray*}
\|y\|_{L^2_\lambda(0,\infty;\R^d)}&=& \left[\int^\infty_0 e^{\lambda t}\left|\int^t_0 e^{(A-\lambda I)(t-s)}f(s)ds\right|^2 dt\right]^{1/2}\\
&\leq & M \left[\int^\infty_0 e^{\lambda t}\left(\int^t_0 e^{(\sigma-\lambda)(t-s)}|f(s)|ds\right)^2 dt\right]^{1/2}\\
&=& M \left[\int^\infty_0\left(\int^t_0 e^{\left(\sigma-\frac{\lambda}{2}\right)(t-s)} e^{\frac{\lambda s}{2}}|f(s)|ds\right)^2 dt\right]^{1/2},
\end{eqnarray*}
and by the Young's inequality
\[
\|y\|_{L^2(0,\infty;\R^d)}\leq M\int^\infty_0 e^{\left(\sigma-\frac{\lambda}{2}\right)}dt\|f\|_{L^2_\lambda(0,\infty;\R^d)}=\frac{2M}{\lambda-2\sigma}\|f\|_{L^2_\lambda(0,\infty;\R^d)}.
\]
Moreover, by \eqref{dsf}
\[
\|\dot y\|_{L^2(0,\infty;\R^d)}\leq\left[\frac{2M}{\lambda-2\sigma}(\|A\|+\lambda)+1\right]\|f\|_{L^2_\lambda(0,\infty;\R^d)}
\]
and hence $y\in W^{1,2}_{\lambda,0}$ as desired. Thus $T$ is surjective and is clearly injective, and the same holds for $T^*$. By the closed range theorem there exists a constant $C>0$ such that
\begin{equation}\label{Tstar}
\|\vp\|_{L^2_\lambda(0,\infty;\R^d)}\leq C\|T^*\vp\|_{\left(W^{1,2}_{\lambda,0}\right)^*}\ \text{for all}\ \vp\in L^2_\lambda(0,\infty;\R^d).
\end{equation}
Let $g\in L^2_\lambda(0,\infty;\R^d)$ be arbitrary and choose $\vp\in L^2_\lambda(0,\infty;\R^d)$ such that $T^*\vp=g$. Then by \eqref{Tstar}
\[
\|\vp\|_{L^2_\lambda(0,\infty;\R^d)}\leq C\|g\|_{\left(W^{1,2}_{\lambda,0}\right)^*}\leq C\tilde C\|g\|_{L^2_\lambda(0,\infty;\R^d)},
\]
where $\tilde C$ is the embedding constant of $L^2_\lambda(0,\infty;\R^d)$ into $\left(W^{1,2}_{\lambda,0}\right)^*$. Moreover $\dot \vp=-A^T\vp-g$ and thus
\[
\|\dot \vp\|_{L^2_\lambda(0,\infty;\R^d)}\leq (C\tilde C+1)\|g\|_{L^2_\lambda(0,\infty;\R^d)}.
\]
It follows that $\vp\in W^{1,2}_\lambda$ and $\lim_{t\ra\infty}\vp(t)=0$. Thus $\vp$ is the solution to \eqref{adjointg} and the claim of the lemma follows.
\end{proof}
}

Since the controls in $U^\Delta$ are piecewise constant functions, we consider  at first  the optimal control on each time interval $I_k$, $k\in\N$.
\begin{prop}\label{propminIk}
Assume that $\lambda>\max\{0,2\sigma\}$. Let $\tilde u\in U^\Delta$ satisfy the following: for any $k\in\N$, $\tilde u(\cdot)\equiv\tilde u_k$ in $I_k$ and
\begin{equation}\label{minIk}
\tilde u_k\in\mathop{arg\,min}_{u=(u_1,\ldots,u_m)\in U}\left\{\int_{I_k}\langle\tilde \vp(t), B u \rangle dt +\gamma b_k\|u\|_p^q\right\},
\end{equation}
where $\tilde y$ is the corresponding trajectory and $\tilde \vp$ is the adjoint state associated to $(\tilde y, \tilde u)$.
Further  for any arbitrary $\omega\in\R^m$ such that $\omega+\tilde u_k\in U$, we define the perturbed control
\[
u^\omega(t):=\left\{
\begin{array}{ll}
 \tilde u_k+\omega & \text{if}\ t\in I_k,\\
 \tilde u(t) & \text{otherwise}.
\end{array}
\right.
\]
Then it holds that $J^\Delta(x,u^\omega)\geq J^\Delta (x,\tilde u)$.
\end{prop}
\begin{proof}
Let $y^\om$ be the trajectory associated with $u^\om$. Then $y^\om-\tilde y$ satisfies
\[
\left\{
\begin{array}{ll}
 \dot y^\om(t)-\dot{\tilde y}(t)=A(y^\om(t)-y(t))+B\omega\mathbbm{1}_{I_k}(t) & \text{for}\ t\in(0,\infty),\\
 y^\om(0)-\tilde y(0)=0.
\end{array}
\right.
\]
Assumption \eqref{minIk} implies that
\begin{equation}\label{EqminIk}
\int_{I_k}\langle\tilde \vp(t), B (\tilde u_k+\om) \rangle dt +\gamma b_k\|\tilde u_k+\om\|_p^q \geq \int_{I_k}\langle\tilde \vp(t), B \tilde u_k \rangle dt +\gamma b_k\|\tilde u_k\|_p^q.
\end{equation}
By the definition of $u^\om$, \eqref{EqminIk} is equivalent to
\begin{equation}\label{Eqvpomega}
\int_{I_k} \langle \tilde \vp(t),B\om\rangle dt+\gamma R(u^\om)-\gamma R(\tilde u)\geq 0.
\end{equation}
For almost all $t>0$ we obtain,
\begin{eqnarray*}
&&\frac{d}{dt}\langle \tilde \vp(t), y^\om(t)-\tilde y(t)\rangle\\
&=& \langle \dot{\tilde \vp}(t),y^\om(t)-\tilde y(t)\rangle + \langle \tilde \vp(t),\dot y^\om(t)-\dot{\tilde y}(t)\rangle\\
&=& \langle -A^T\tilde \vp(t)-e^{-\lambda t}(\tilde y(t)-y_d), y^\om(t)-\tilde y(t)\rangle+\langle \tilde \vp(t), A(y^\om (t)-\tilde y(t))+B\om \mathbbm{1}_{I_k}(t)\rangle \\
&=& -e^{-\lambda t}\langle \tilde y(t)-y_d, y^\om(t)-\tilde y(t)\rangle +\langle \tilde \vp(t), B\om \mathbbm{1}_{I_k}(t)\rangle.
\end{eqnarray*}
Note that $\lim_{t\ra\infty}\tilde \vp(t)=0$ and $y^\om(0)-\tilde y(0)=0$, and  therefore
\[
\int^\infty_0 \frac{d}{dt}\langle \tilde \vp(t), y^\om(t)-\tilde y(t)\rangle=0.
\]
Consequently we obtain 
\[
\int^\infty_0 \left[-e^{-\lambda t}\langle \tilde y(t)-y_d, y^\om(t)-\tilde y(t)\rangle +\langle \tilde \vp(t), B\om \mathbbm{1}_{I_k}(t)\rangle\right]dt=0,
\]
i.e.,
\begin{equation}\label{EqBomega}
\int^\infty_0 e^{-\lambda t}\langle \tilde y(t)-y_d, y^\om(t)-\tilde y(t)\rangle dt=\int_{I_k} \langle \tilde \vp(t), B\om\rangle.
\end{equation}
To compute the left-hand side of \eqref{EqBomega}, we have for every $t>0$
\begin{equation}\label{Eqsquare}
\|y^\om(t)-y_d\|^2_2-\|\tilde y(t)-y_d\|^2_2=\|y^\om(t)-\tilde y(t)\|^2_2+2\langle y^\om(t)-\tilde y(t),\tilde y(t)-y_d\rangle,
\end{equation}
and now \eqref{Eqvpomega}, \eqref{EqBomega} and \eqref{Eqsquare} imply that
\begin{align*}
&&\int^\infty_0 \frac{1}{2}e^{-\lambda t}\|y^\om(t)-y_d\|^2_2dt- \int^\infty_0 \frac{1}{2}e^{-\lambda t}\|\tilde y(t)-y_d\|^2_2dt \\ &&-\int^\infty_0 \frac{1}{2}e^{-\lambda t}\|y^\om(t)-\tilde y(t)\|^2_2dt+\gamma R(u^\om)-\gamma R(\tilde u)\geq 0.
\end{align*}
Then we deduce that
\[
J^\Delta (x,u^\om)-J^\Delta (x,\tilde u)\geq \int^\infty_0 \frac{1}{2}e^{-\lambda t}\|y^\om(t)-\tilde y(t)\|^2_2dt \geq 0,
\]
which ends the proof.
\end{proof}
Proposition \ref{propminIk} provides the way to construct optimal controls on each $I_k$, and this  procedure can be naturally extended to construct globally optimal controls.
\begin{thm}\label{thmminimizer}
Assume that $\lambda>\max\{0,2\sigma\}$. Let $\bar u\in U^\Delta$ satisfy the following: for any $k\in \N$ and $t\in I_k$,
\begin{equation}\label{umin}
\bar u(t)\in\mathop{arg\,min}_{u=(u_1,\ldots,u_m)\in U}\left\{\int_{I_k}\langle\bar \vp(t), B u \rangle dt +\gamma b_k\|u\|_p^q\right\},
\end{equation}
where $\bar y$ is the corresponding trajectory and $\bar \vp$ is the adjoint state associated to $(\bar y, \bar u)$.
Then $\bar u \in U^\Delta$ is a minimizer of problem \eqref{ocpdtd}, i.e.
\[
J^\Delta (x,u)\geq J^\Delta (x,\bar u),\ \forall\,u\in U^\Delta.
\]
\end{thm}
\begin{proof}
For any $u\in U^\Delta$, we define a sequence $(u^n)_{n\in\N}$ by
\[
u^n(t):=\left\{
\begin{array}{ll}
 \bar u(t) & \text{if}\ t\in I_k,\ k=0,\ldots,n,\\
 u(t) & \text{otherwise}.
\end{array}
\right.
\]
Therefore, $u^n\ra \bar u$ pointwise in $[0,\infty[$. Let $y^n$ be the trajectory associated with $u^n$. By the same argument as in Theorem \ref{thmexistence}, we deduce that
\[
y^n\ra \bar y\ \text{pointwise in}\ [0,\infty[.
\]
Assumption \eqref{umin} and Proposition \ref{propminIk} imply that
\[
J^\Delta (x,u)\geq J^\Delta (x,u^0)\geq J^\Delta (x,u^1)\geq \cdots \geq J^\Delta (x,u^n),\ \forall\,n\in\N.
\]
By Fatou's Lemma,
\begin{eqnarray*}
J^\Delta (x,u)&\geq& \mathop{\lim\,\inf}_{n\ra\infty} \left\{\int^\infty_0 \frac{1}{2} e^{-\lambda t}\|y^n(t)-y_d\|^2_2 dt +\gamma R(u^n)\right\}\\
&\geq& \int^\infty_0 \frac{1}{2} e^{-\lambda t}\|\bar y(t)-y_d\|^2_2 dt +\gamma R(\bar u)\\
&=& J^\Delta (x,\bar u),
\end{eqnarray*}
and  we conclude that $\bar u$ is a minimizer of problem \eqref{ocpdtd}.
\end{proof}

Based on the optimality conditions \eqref{umin}, similar results on sparsity and switching properties as Theorem \ref{THMswitching} can be deduced by the same arguments as in the proof of Theorem \ref{THMswitching}.
\begin{thm}
Following the same assumptions and notations in Theorem \ref{thmminimizer}, we set
\[
\varphi_k=\int_{I_k}B^T{\color{black}\bar \varphi(t)}dt,\ \gamma_k=\gamma b_k.
\]
For each $k\in\N$, we define the following index sets:
\[
I^-_k=\{i\in\{1,\ldots,m\}\,:\, |\vp_{k,i}|\rho_i^{1-q}<\gamma_k \},
\]
\[
I^0_k=\{i\in\{1,\ldots,m\}\,:\, |\vp_{k,i}|\rho_i^{1-q}=\gamma_k \},
\]
\[
I^+_k=\{i\in\{1,\ldots,m\}\,:\, |\vp_{k,i}|\rho_i^{1-q}>\gamma_k \},
\]
The following properties hold:
\begin{enumerate}[(i)]
\item
For $k\in\N$, $t\in I_k$ and $i\in I^-_k$,
\[
\bar u_i(t)=0.
\]
\item
For $k\in\N$, $t\in I_k$ and $i\in I^0_k$,
\[
\left\{
\begin{array}{lllll}
\bar u_i(t)=0, & \text{if}\ I^+_k\neq\emptyset,\\[1.7ex]
\bar u_i(t)\in\{0,-\rho_i\mathop{sgn}(\vp_{k,i})\}, & \\
\bar u_i(t) \bar u_{j}(t)=0,\ i,j\in I^0_k,\ i\neq j, & \text{if}\ I^+_k=\emptyset,\ q\in[p,1[,\\[1.7ex]
\bar u_i(t)\in [0,-\rho_i\mathop{sgn}(\vp_{k,i})], & \\
\bar u_i(t) \bar u_{j}(t)=0,\ i,j\in I^0_k,\ i\neq j, & \text{if}\ I^+_k=\emptyset,\ q=1.
\end{array}
\right.
\]
\item
For $k\in\N$, $t\in I_k$ and $i\in I^+_k$, we have
\[
\bar u_i(t)\in\{0,-\rho_i\mathop{sgn}(\vp_{k,i})\},
\]
and $\max_{i I_k^+}|\bar u_i(t)| \neq 0$.
\end{enumerate}
\end{thm}

\section{Numerical experiments}
In this section we present numerical experiments for the computation of optimal control laws for the problem
\begin{equation*}
\underset{u(\cdot)\in L^{\infty}(0,\infty;U_\infty)}{\inf}J(x,u):=\int^{\infty}_0 e^{-\lambda t}\left(\frac{1}{2}\|y(t)-y_d\|^2_2+\gamma\|u(t)\|_p^q\right)dt,
\end{equation*}
constrained to the nonlinear dynamical system
\begin{equation*}
\left\{
\begin{array}{ll}
\dot y(s)=f(y(s),u(s)):=f_0(y(s))+\sum^m_{i=1}f_i(y(s))u_i(s) & \text{in}\ ]0,\infty[,\\
y(0)=x.
\end{array}
\right.
\end{equation*}
For the realization of globally optimal control laws we proceed as in \cite{KKR17}, i.e. by following a dynamic programming approach. The value function $V(x):=\inf J(x,u)$ associated to this infinite horizon optimal control problem satisfies the following first order Hamilton-Jacobi-Bellman equation
\begin{equation*}
\lambda V(x)+\underset{u\in U_\infty}{\sup}\{-f(x,u)\cdot\nabla V(x)-\frac{1}{2}\|x-y_d\|^2_2-\gamma\|u\|_p^q\}=0\,,
\end{equation*}
which leads to the optimal feedback map
\begin{equation} \label{eq6.30}
\bar{u}(x):=\underset{u\in U_\infty}{\text{argmin}}\left\{f(x,u)\cdot\nabla V(x)+\gamma\|u\|_p^q \right\}\,.
\end{equation}
The solution of the Hamilton-Jacobi-Bellman equation and of the optimal feedback mapping are numerically approximated by a first-order semi-Lagrangian scheme with policy iteration as discussed in \cite{Alla15}. The well-posedness of this numerical scheme is guaranteed under boundedness and continuity assumptions for the dynamics $f(x,u)$ and the cost. Convergence of controls, however, is only guaranteed for convex running costs. Nevertheless, the results we report indicate that the semi-Lagrangian scheme converges to optimal controls exhibiting the expected sparsity and switching properties. This scheme has also been applied to the solution of sparse optimal feedback control problems in \cite{Albi17,Falcone14}. In the case $p=q=1$ the minimization operation in \eqref{eq6.30} can be realized by means of semismooth Newton methods as \cite{KKK16}. For different values of $p$ and $q$, the minimizer is chosen by discretizing the control set $U_\infty$ into a finite number of values and making a pointwise evaluation of the Hamiltonian.

\subsubsection*{Eikonal dynamics}

We begin by considering eikonal-type dynamics for planar motion of the form
\begin{align*}
\dot x_1(s)&=u_1(s)\\
\dot x_2(s)&=u_2(s)\,,
\end{align*}
where $|u_i(s)|\leq 0.5$ for $i=1,2.$ The state space is set to be $\Omega=[-1,1]^2$, the discount factor $\lambda=0.2$, and $\gamma=1$. The goal is to drive the state to the origin, and therefore $y_d=(0,0)$. The optimal control fields in the state space for different $p,q$ values are shown in Figure \ref{eikcon}.
\begin{figure}[!ht]
	\centering
	\subfloat[$p=1, q=1$]{\label{fig:t1b}\includegraphics[width=0.33\textwidth]{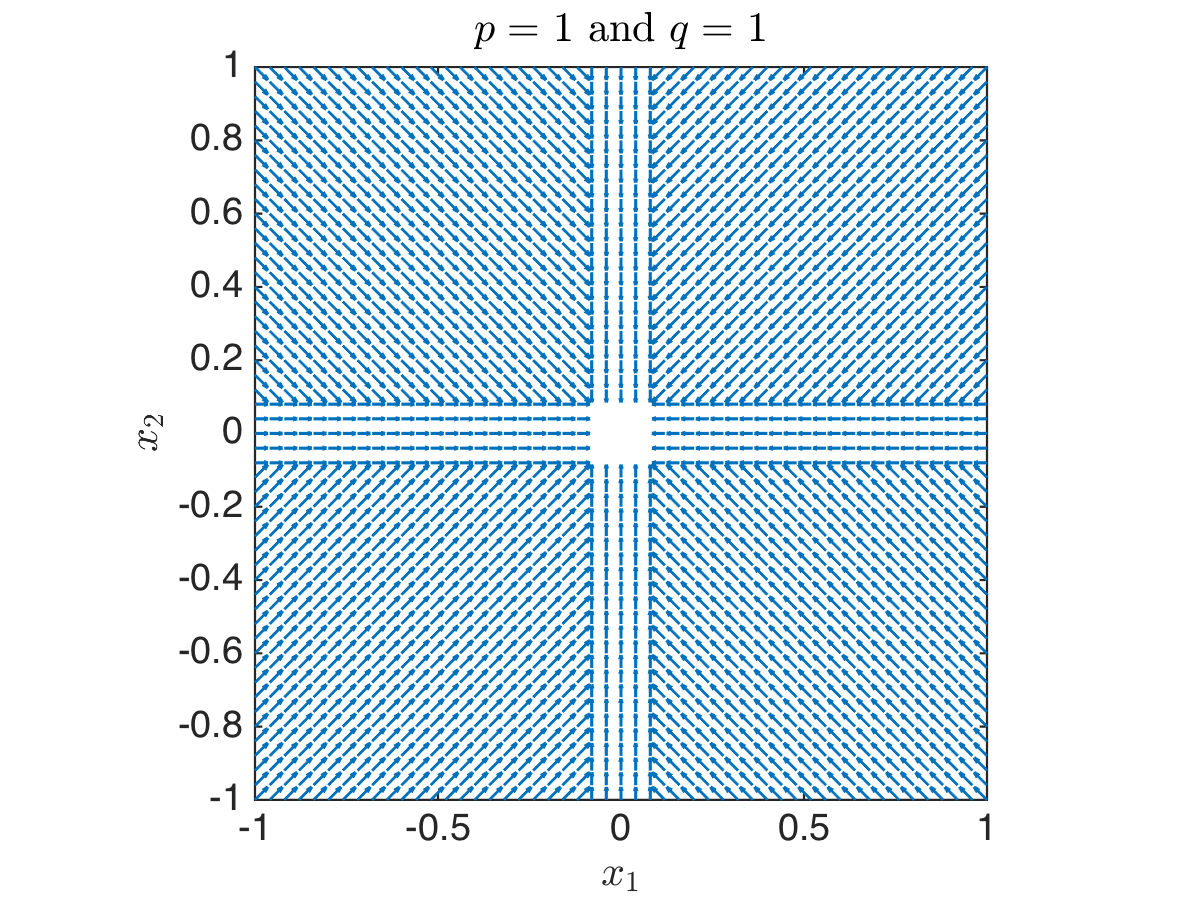}}\hfill
	\subfloat[$p=0.8, q=1$]{\label{fig:t1c}\includegraphics[width=0.33\textwidth]{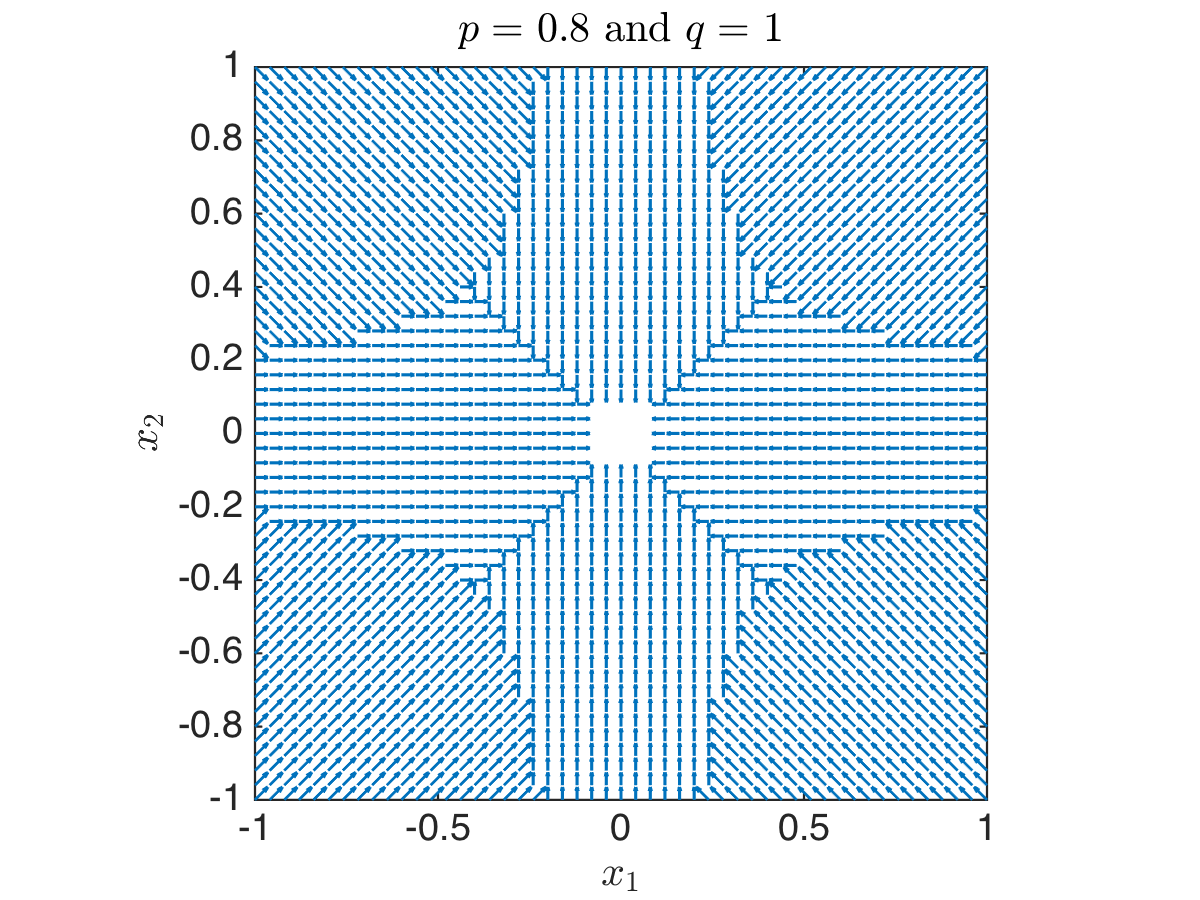}}\hfill
	\subfloat[$p=0.6, q=1$]{\label{fig:t1d}\includegraphics[width=0.33\textwidth]{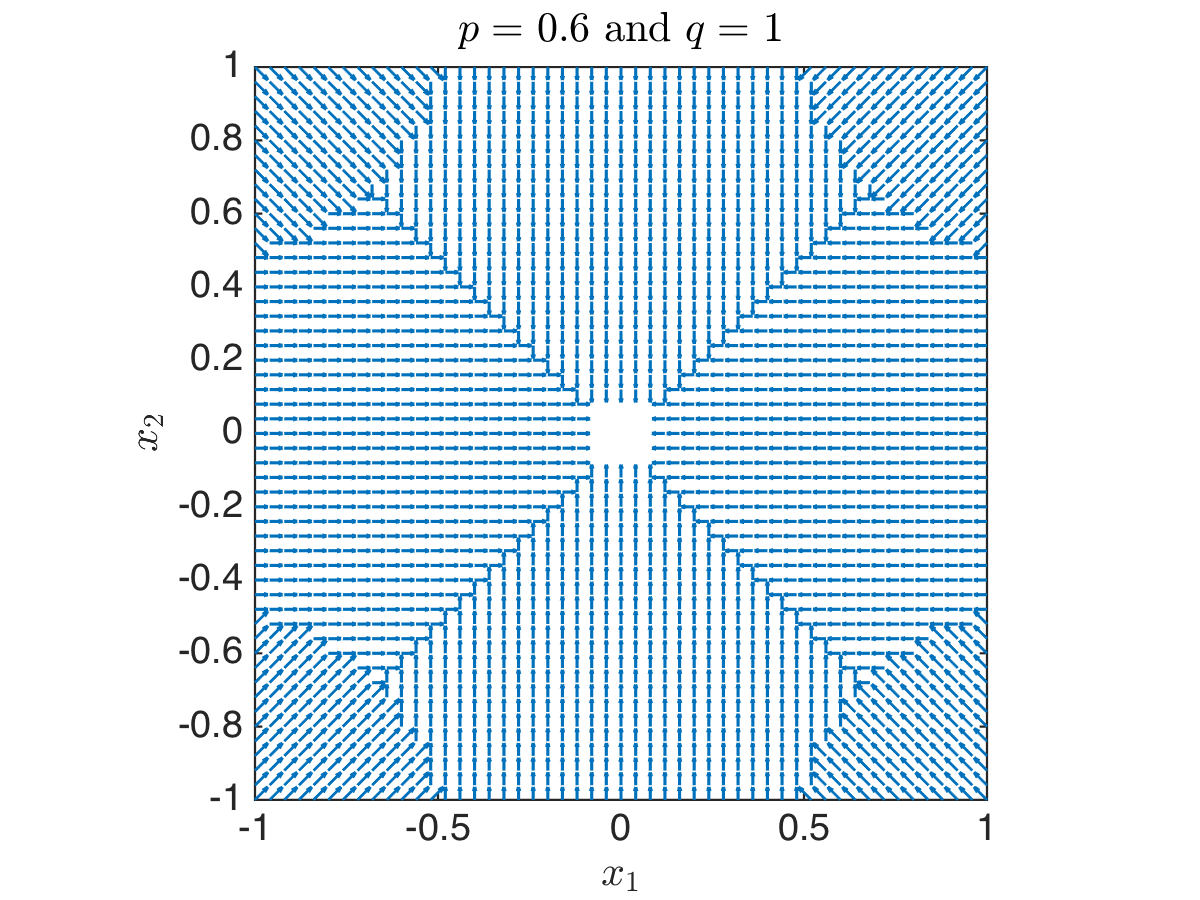}}\\
	\subfloat[$p=0.2, q=0.2$]{\label{fig:t1e}\includegraphics[width=0.33\textwidth]{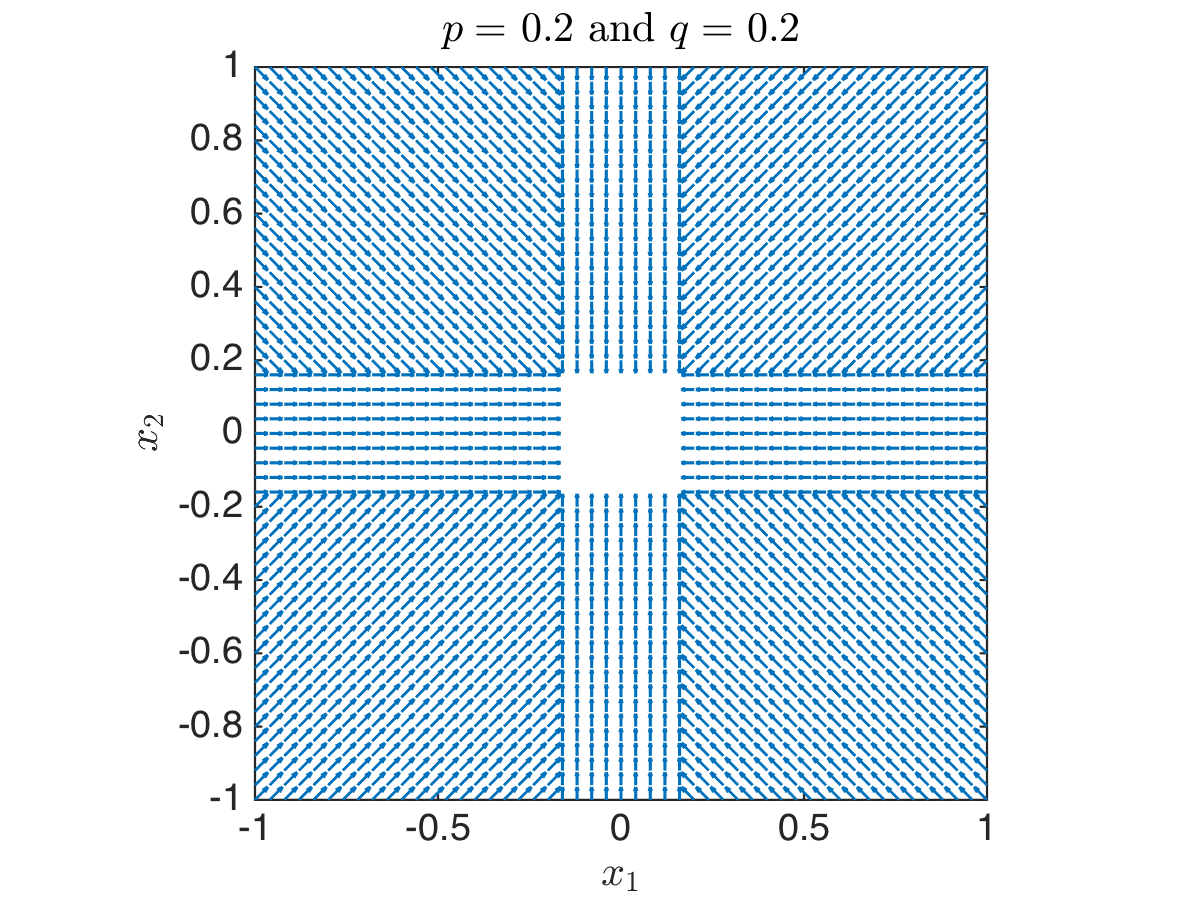}}\hfill
	\subfloat[$p=0.2, q=0.3$]{\label{fig:t1f}\includegraphics[width=0.33\textwidth]{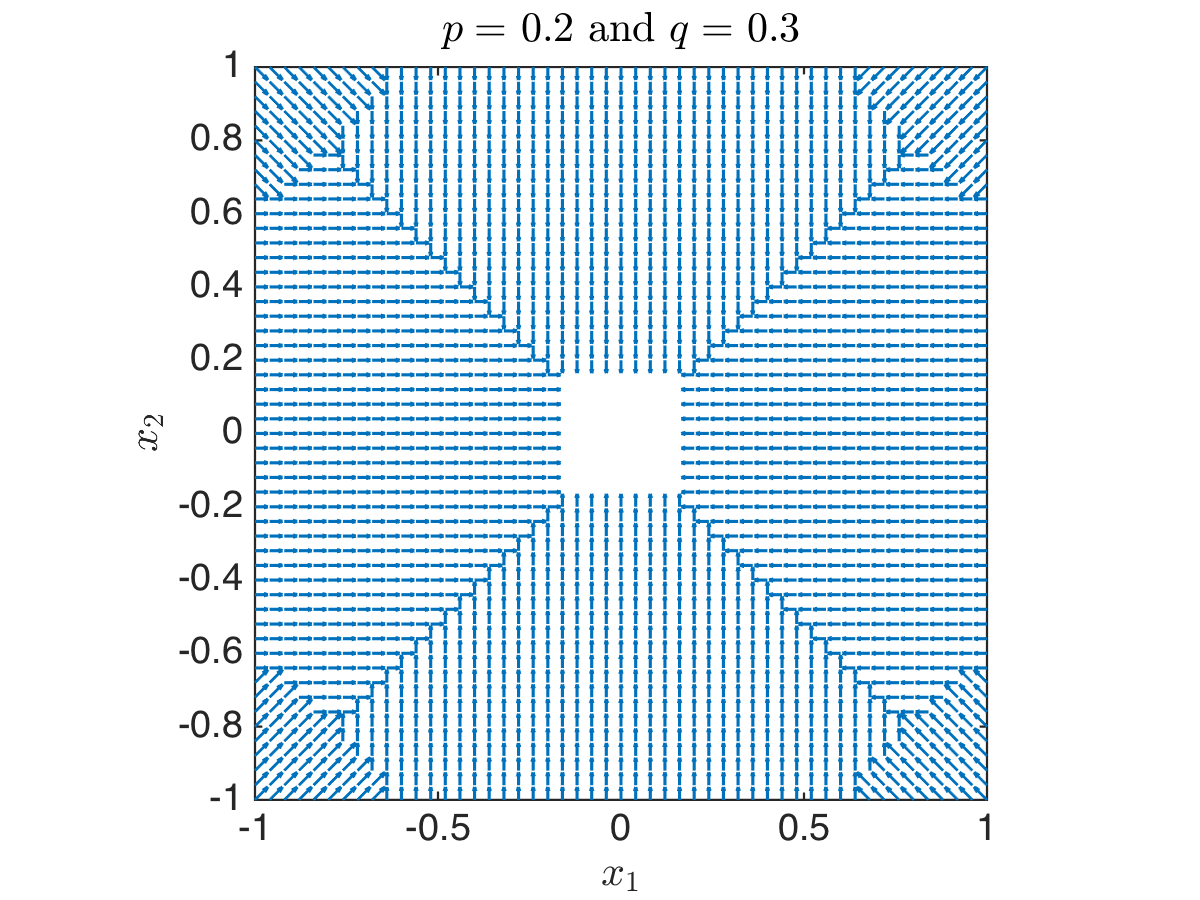}}\hfill
	\subfloat[$p=0.2, q=0.8$]{\label{fig:t1g}\includegraphics[width=0.33\textwidth]{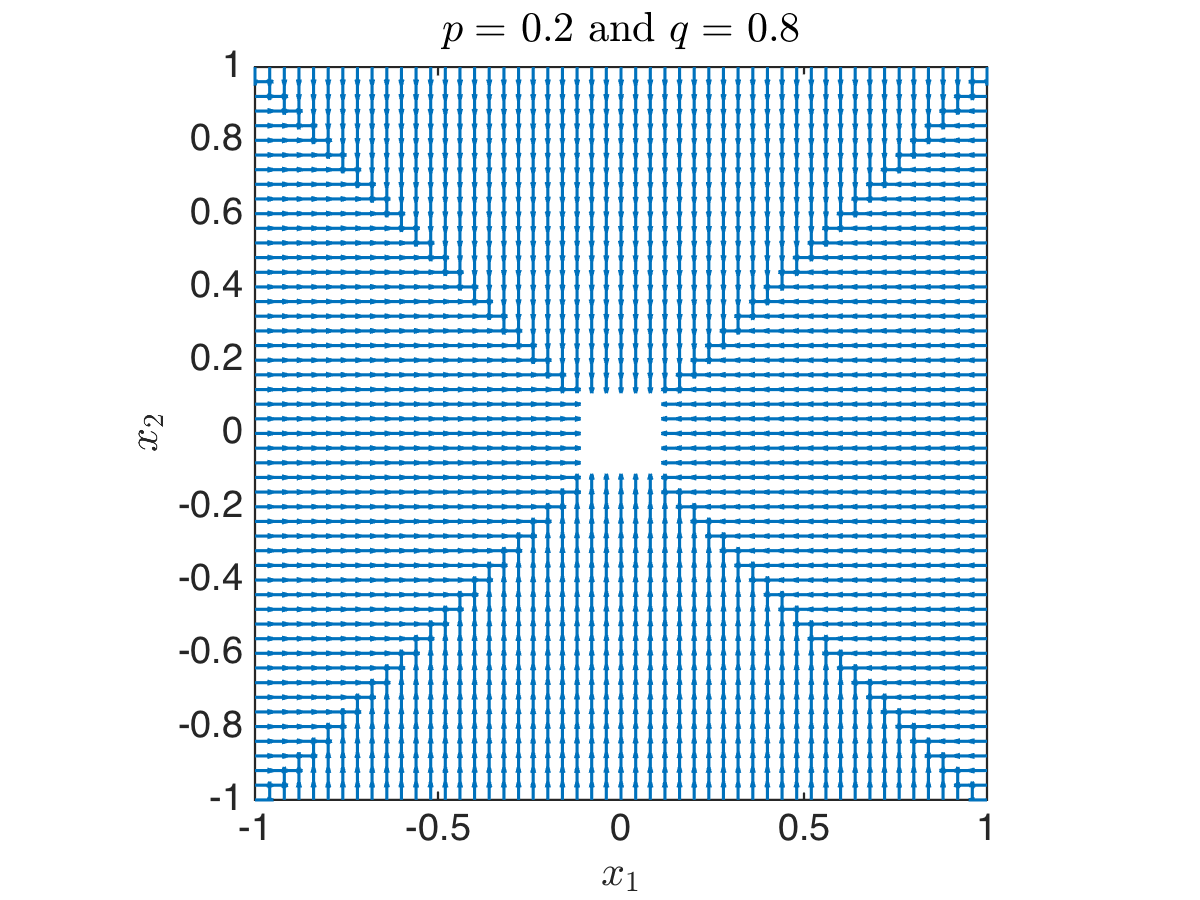}}
	\caption{Eikonal dynamics, optimal control fields for different control penalizations$\|u\|_p^q$.}\label{eikcon}
\end{figure}

We observe the following:
\begin{itemize}
	\item[a)] The  case $p=q=1$ has been already reported in \cite{KKR17}.  There exists a switching band of width $\gamma\lambda$, where the optimal control points unidirectionally towards the origin, and $\bar{u}=0$ for $\|u\|_{\infty}\leq\gamma\lambda$.
	\item[b,c)] Departing from $p=q=1$ and reducing the value of $p$, a switching region with only one active control component arises. It increases as the  ratio $q/p$ increases. Note that for $q=1$, the region where $\bar{u}=0$ remains unchanged.
	\item[d)] The switching and the sparsity regions are larger for $p=q=0.2$ than for $p=q=1$. Only in the particular case $\rho=1$ these regions would remain the same.
	\item[e,f)] Increasing the $q/p$ ratio by departing from smaller values of $q$ generates a larger switching region, leading to a fully switching controller for a ratio of $q/p$ sufficiently large. Note that increasing $q/p$ for $q\neq 1$ also leads to a decrease of the sparsity region.
\end{itemize}

\subsubsection*{Nonlinear dynamics of a double-well potential}
We now address the synthesis of optimal controllers for nonlinear dynamics.  We consider a system corresponding to a single one-dimensional particle moving in a double-well potential, subject to a controlled damping, and a direct external forcing via
\begin{align*}
\dot x(s)&=v(s)\\
\dot v(s)&= -(1+u_1(s))v(s)+(x(s)-x^3(s)) + u_2(s).
\end{align*}
In the absence of control action ($u_1=u_2=0$), the damped particle has two stable equilibrium positions, namely $x=\pm 1,v=0$  (we drop the state-space notation $(x_1,x_2)$ for $(x,v)$), with their corresponding basins of attraction. Here our goal is to steer the particle to the equilibrium $y_d=(1,0)$. We consider a set of initial conditions in $\Omega=[-2,2]^2$, and set $\gamma=0.1$, $\rho=1$, and $\lambda=0.01$. Optimal controls  are shown in Figure \ref{dw1}.
\begin{figure}[!ht]
	\subfloat[$u_1 ,p=1, q=1$]{\label{fig:t2a}\includegraphics[width=0.33\textwidth]{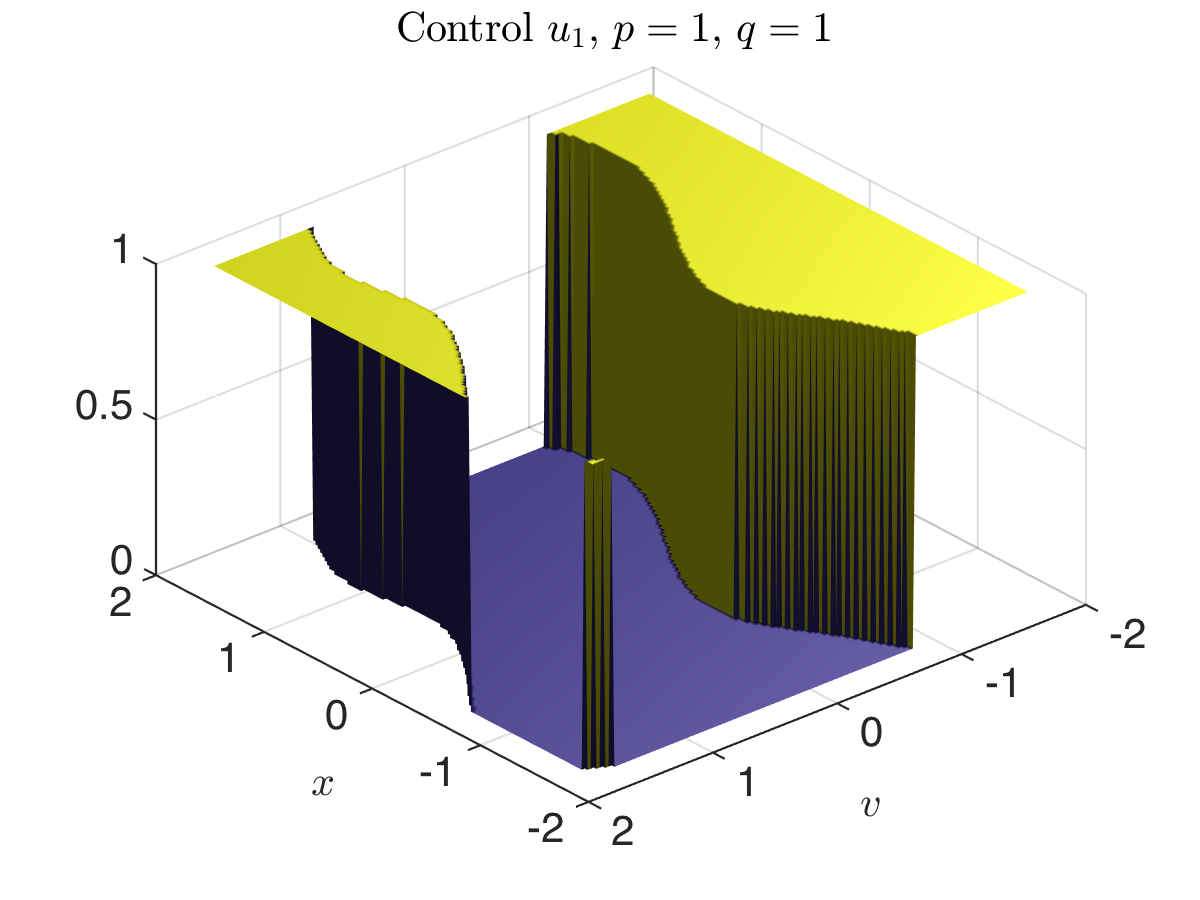}}\hfill
	\subfloat[$u_1, p=0.6, q=1$]{\label{fig:t2b}\includegraphics[width=0.33\textwidth]{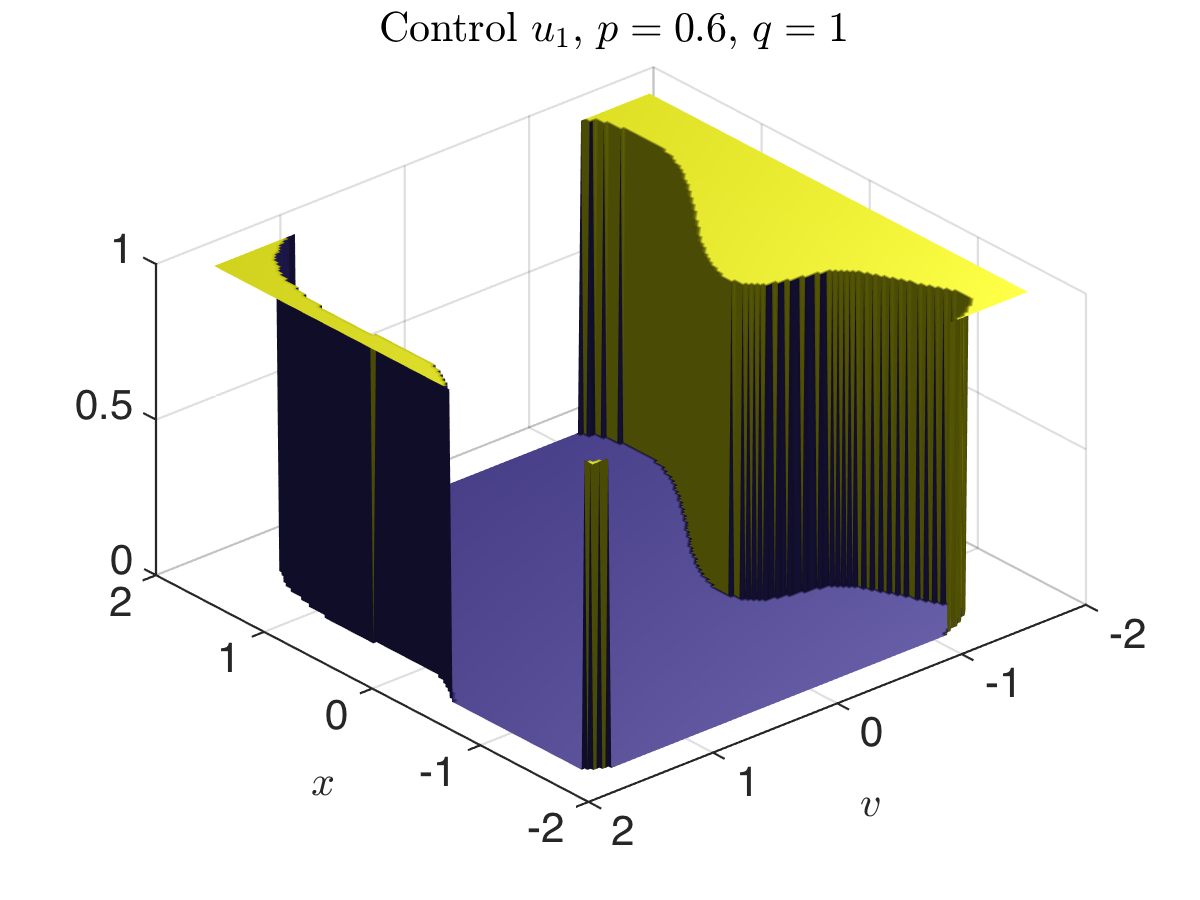}}\hfill
	\subfloat[$u_1, p=0.2, q=1$]{\label{fig:t2c}\includegraphics[width=0.33\textwidth]{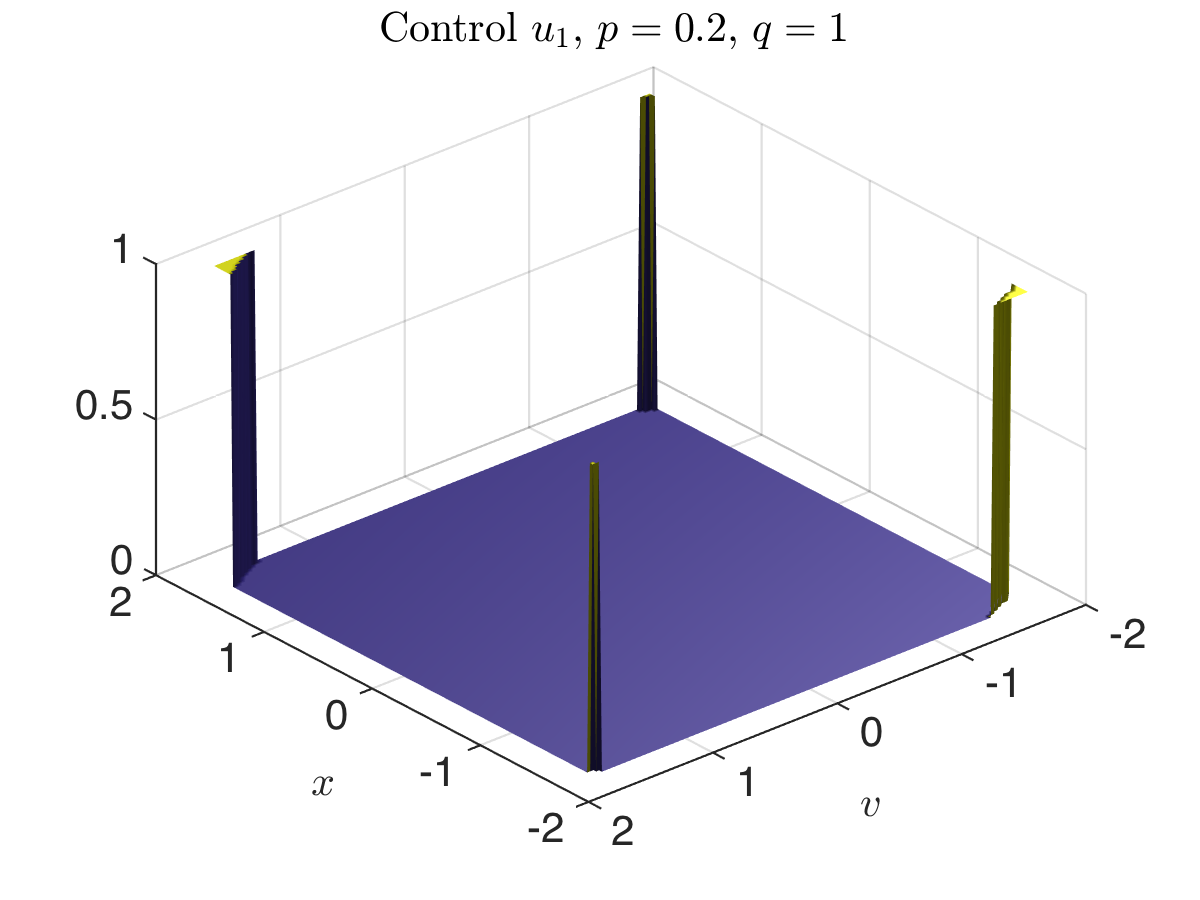}}\\		
	\subfloat[$u_2, p=1, q=1$]{\label{fig:t2d}\includegraphics[width=0.33\textwidth]{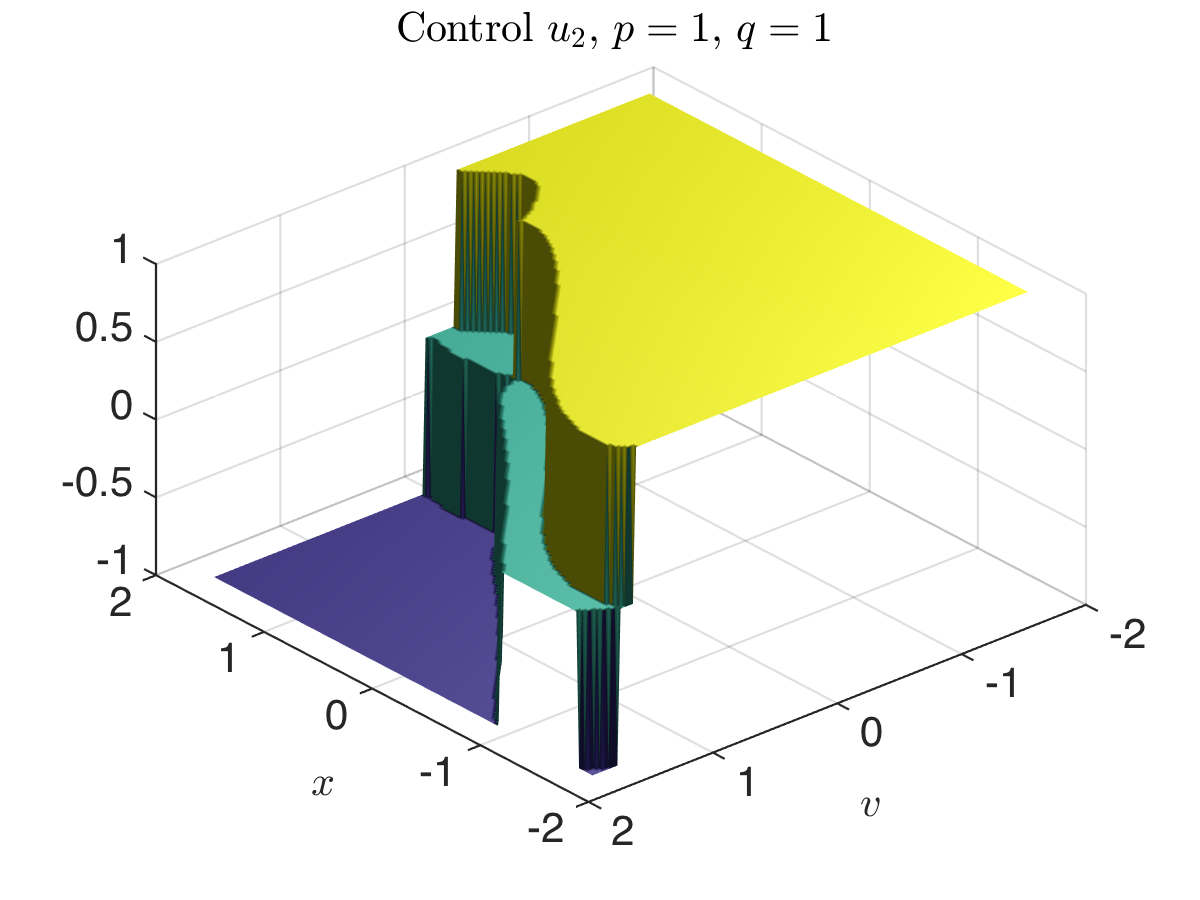}}\hfill
	\subfloat[$u_2, p=0.6, q=1$]{\label{fig:t2e}\includegraphics[width=0.33\textwidth]{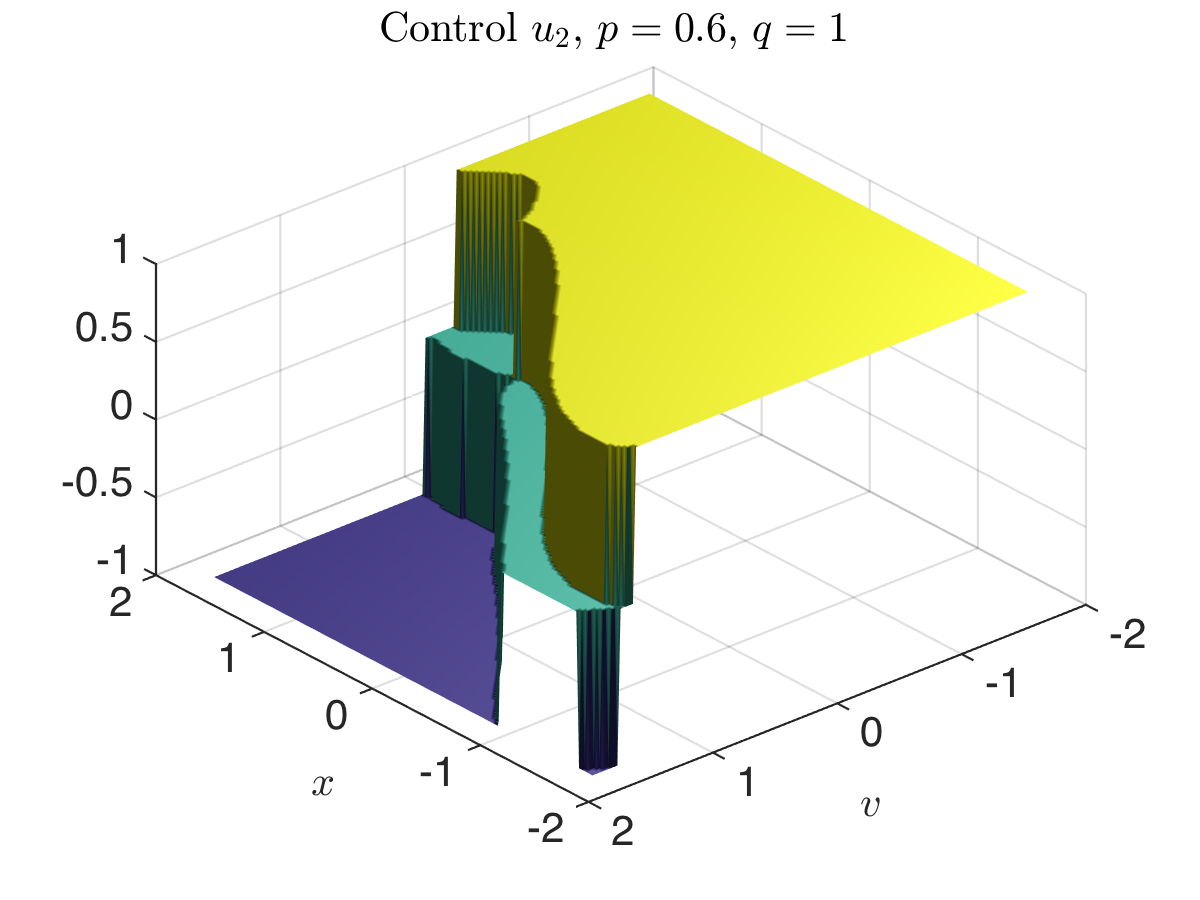}}\hfill
	\subfloat[$u_2, p=0.2, q=1$]{\label{fig:t2f}\includegraphics[width=0.33\textwidth]{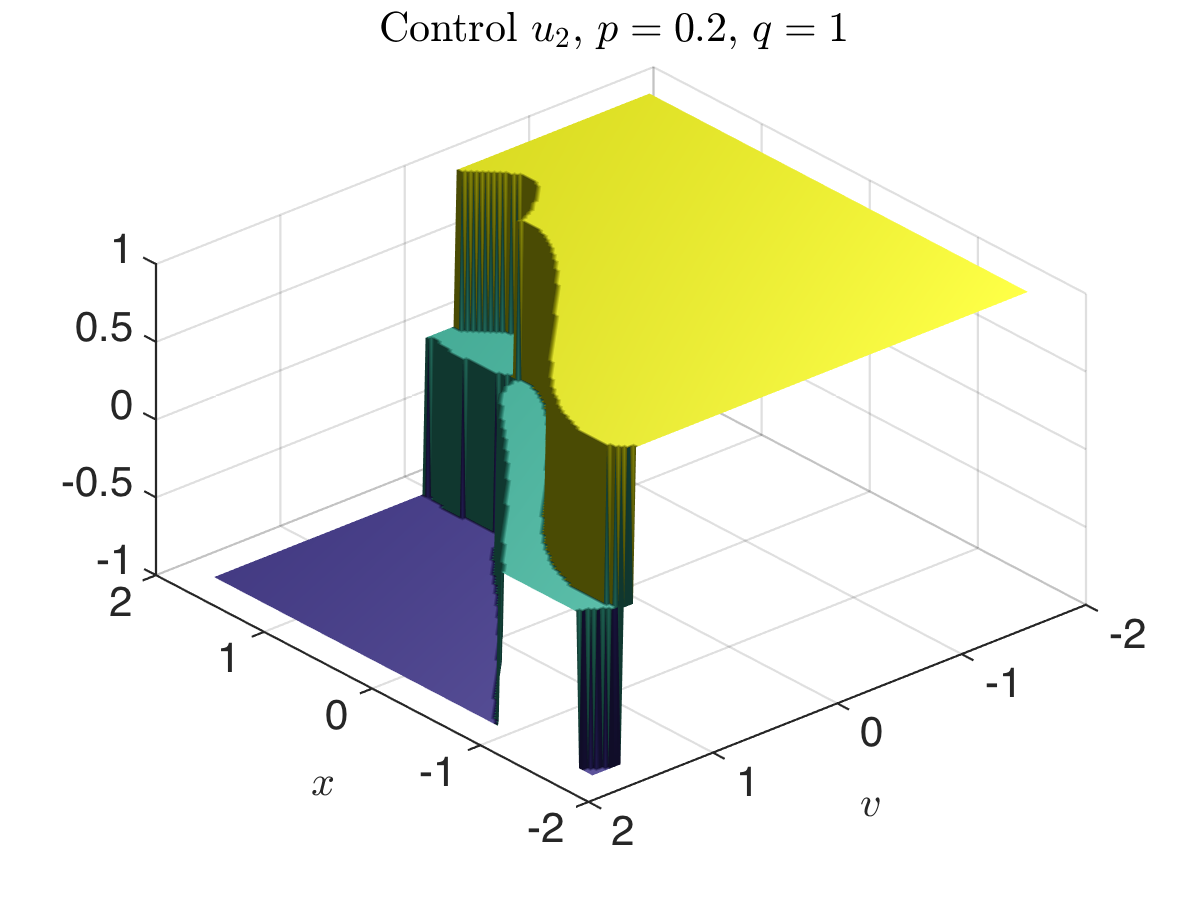}}\\
	\subfloat[$\|u\|_0, p=1, q=1$]{\label{fig:t2g}\includegraphics[width=0.33\textwidth]{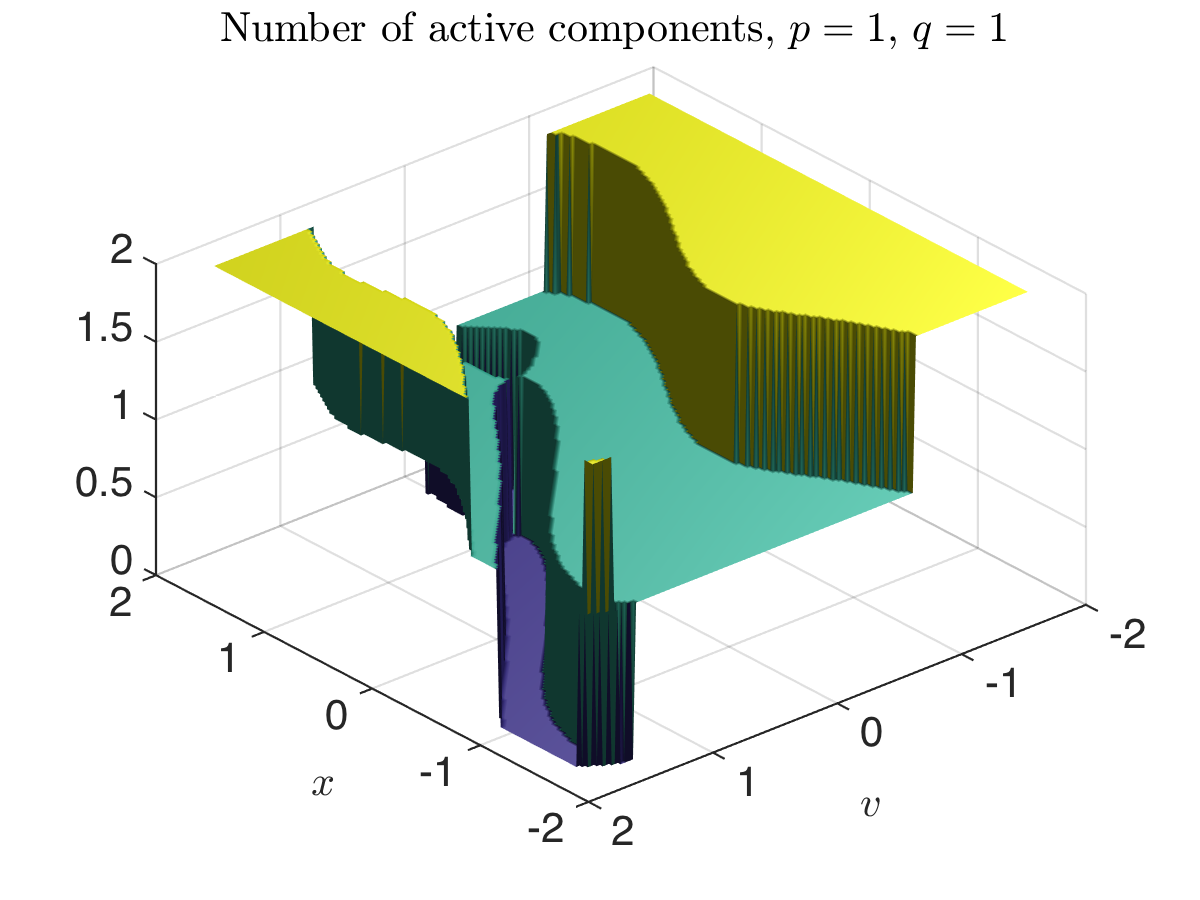}}\hfill
	\subfloat[$\|u\|_0, p=0.6, q=1$]{\label{fig:t2h}\includegraphics[width=0.33\textwidth]{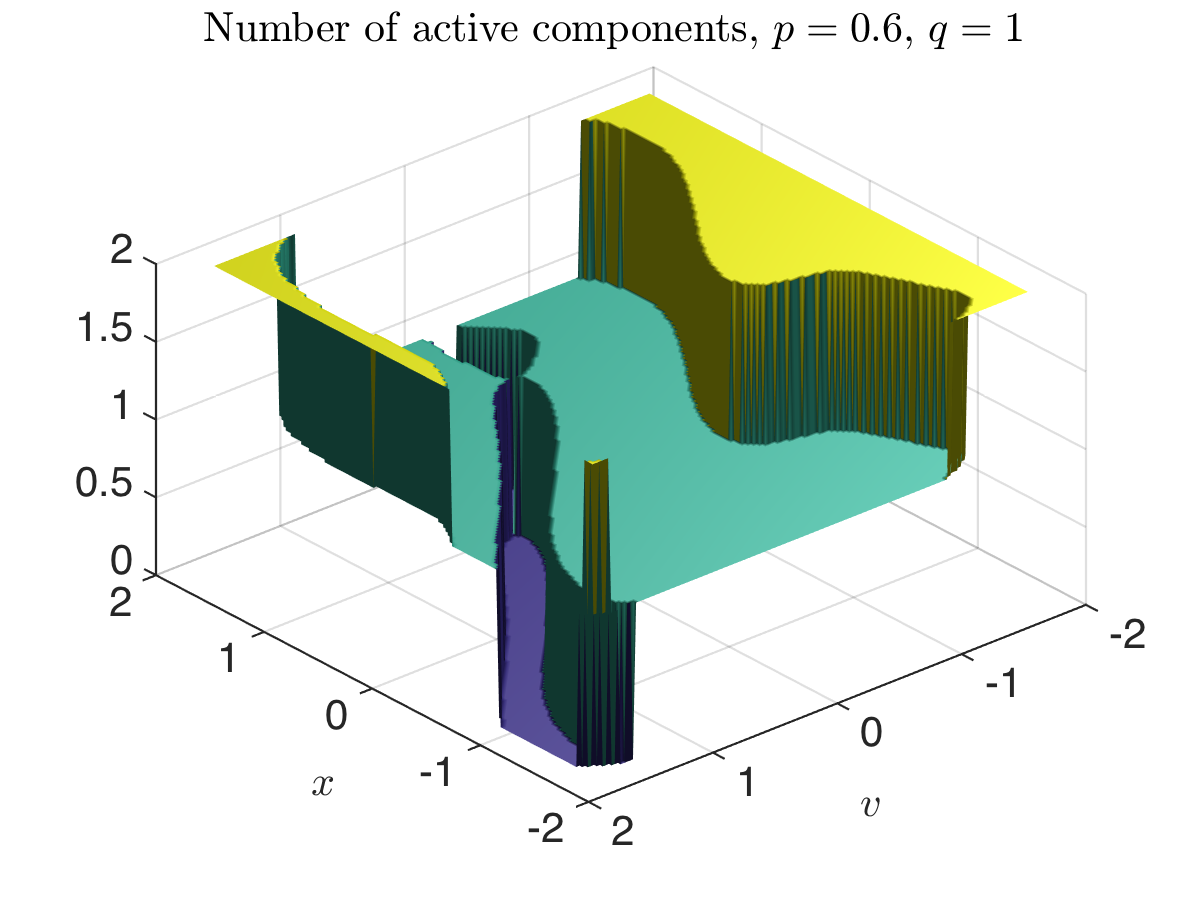}}\hfill
	\subfloat[$\|u\|_0, p=0.2, q=1$]{\label{fig:t2i}\includegraphics[width=0.33\textwidth]{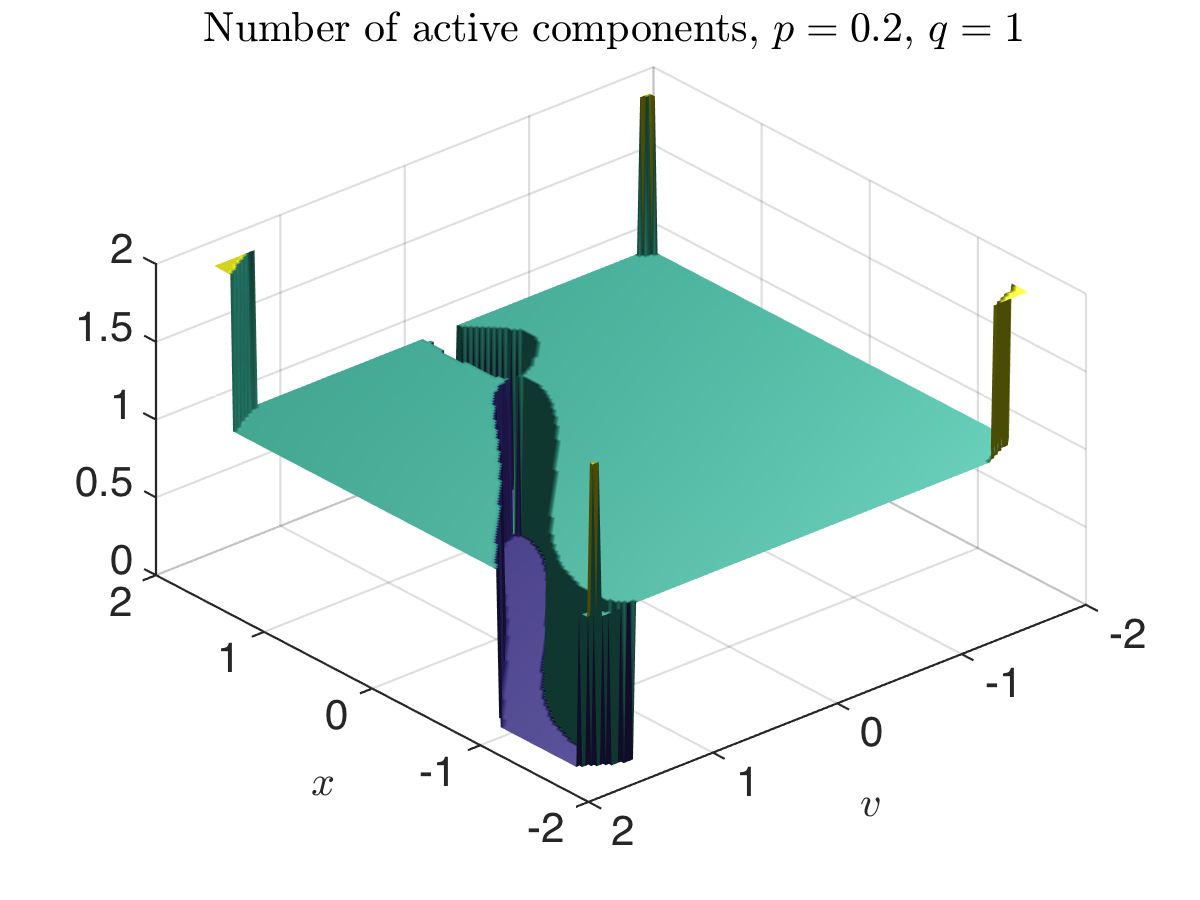}}
	\caption{Optimal controls for the double-well nonlinear control problem. The first two rows show the control variables $u_1$ and $u_2$ for different values of $p$ and $q$.}\label{dw1}
\end{figure}

\noindent We observe:

\begin{itemize}
	\item [a,b,c)] By reducing the value of $p$  with $q=1$, the region where the control $u_1$ is active decreases.
	\item [d,e,f)] Reducing $q$ does not affect the sparsity pattern of $u_2$. The linear control action via $u_2$ is more relevant for the stabilization goal than the bilinear control term $u_1v$. As expected it becomes insignificant as $v$ becomes small.
	\item[g,h,i)] Overall, the reduction of $p$ has  a significant effect on the increase of the switching region.
\end{itemize}

\noindent In order to investigate a setting with a richer interplay between the control variables and the switching structure, we consider a modified version of the double-well control system given by
\begin{align*}
\dot x(s)&=v(s)\\
\dot v(s)&= -(1+u_1(s))v(s)+(x(s)-x^3(s)) + u_2(s)x(s),
\end{align*}
where $u_2$ enters now in a bilinear fashion. The optimal controllers are significantly different compared to the previous setting, as shown in  Figure \ref{dw2}.

\begin{figure}[!ht]
	\subfloat[$u_1 ,p=1, q=1$]{\label{fig:t3a}\includegraphics[width=0.33\textwidth]{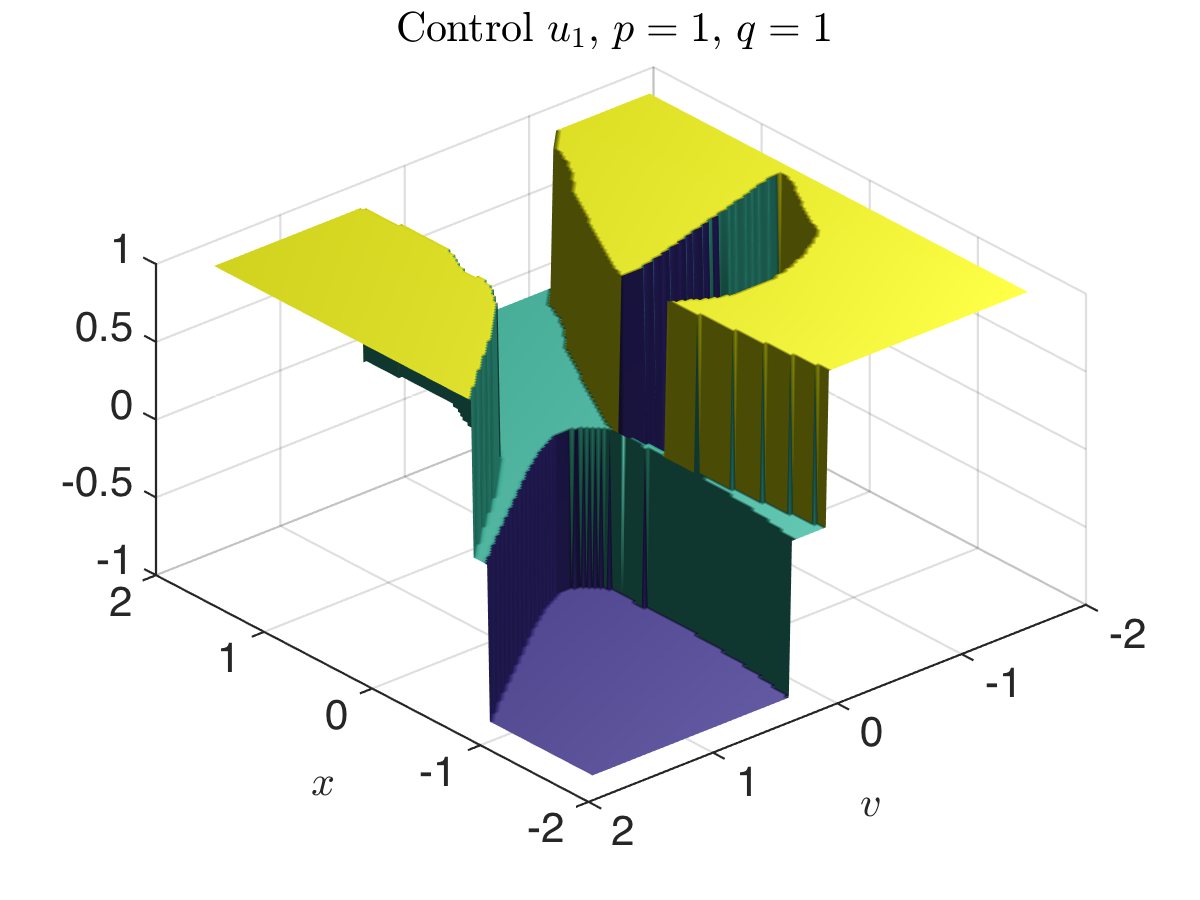}}\hfill
	\subfloat[$u_1 ,p=0.2, q=0.6$]{\label{fig:t3b}\includegraphics[width=0.33\textwidth]{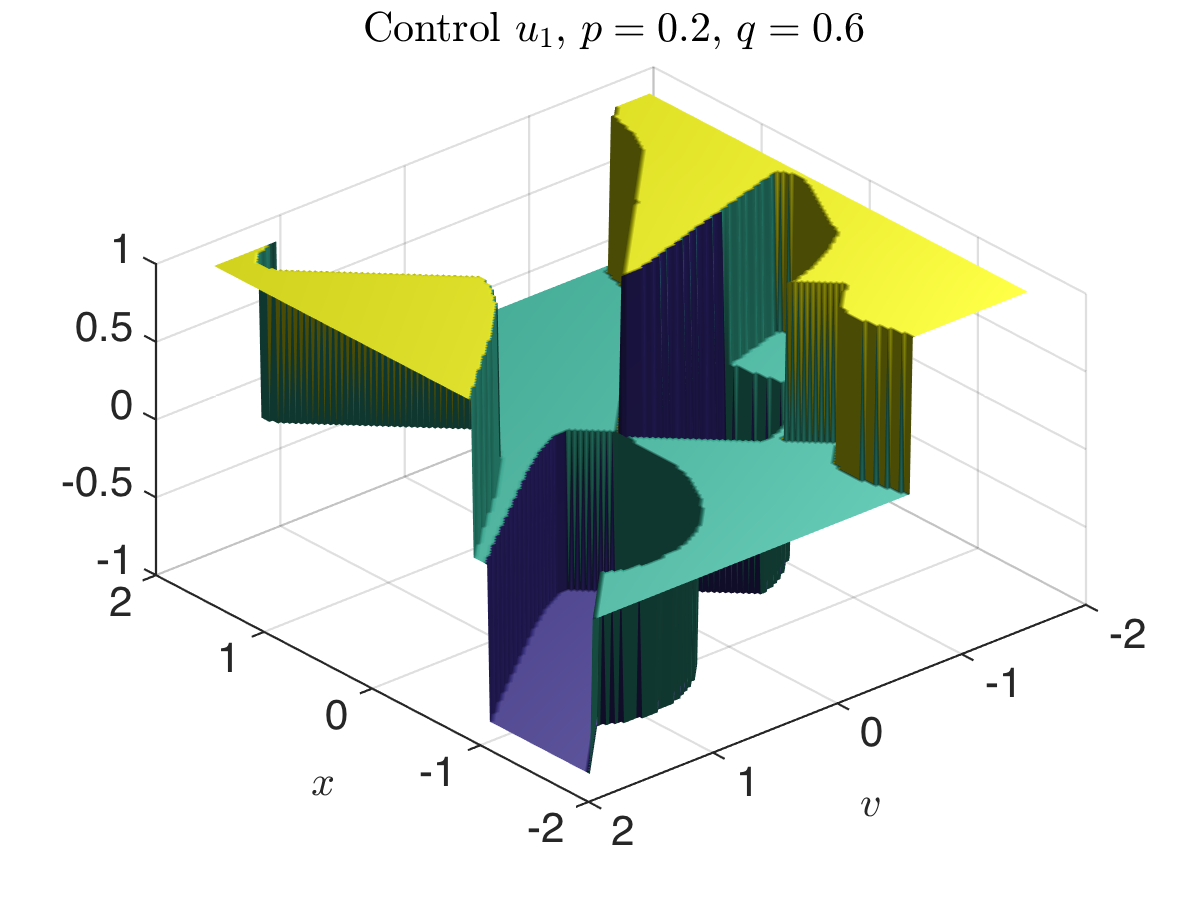}}\hfill
	\subfloat[$u_1 ,p=0.2, q=1$]{\label{fig:t3c}\includegraphics[width=0.33\textwidth]{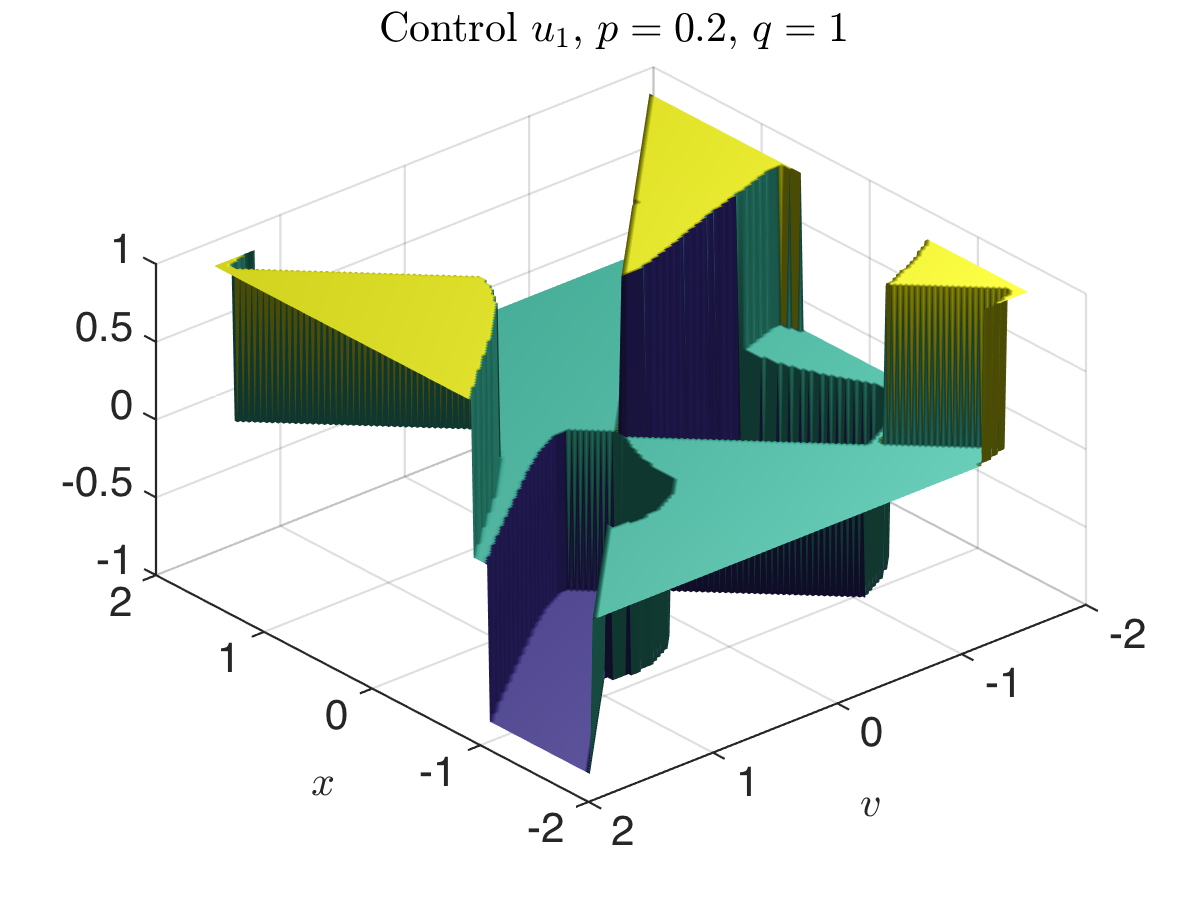}}\\		
	\subfloat[$u_2,p=1, q=1$]{\label{fig:t3d}\includegraphics[width=0.33\textwidth]{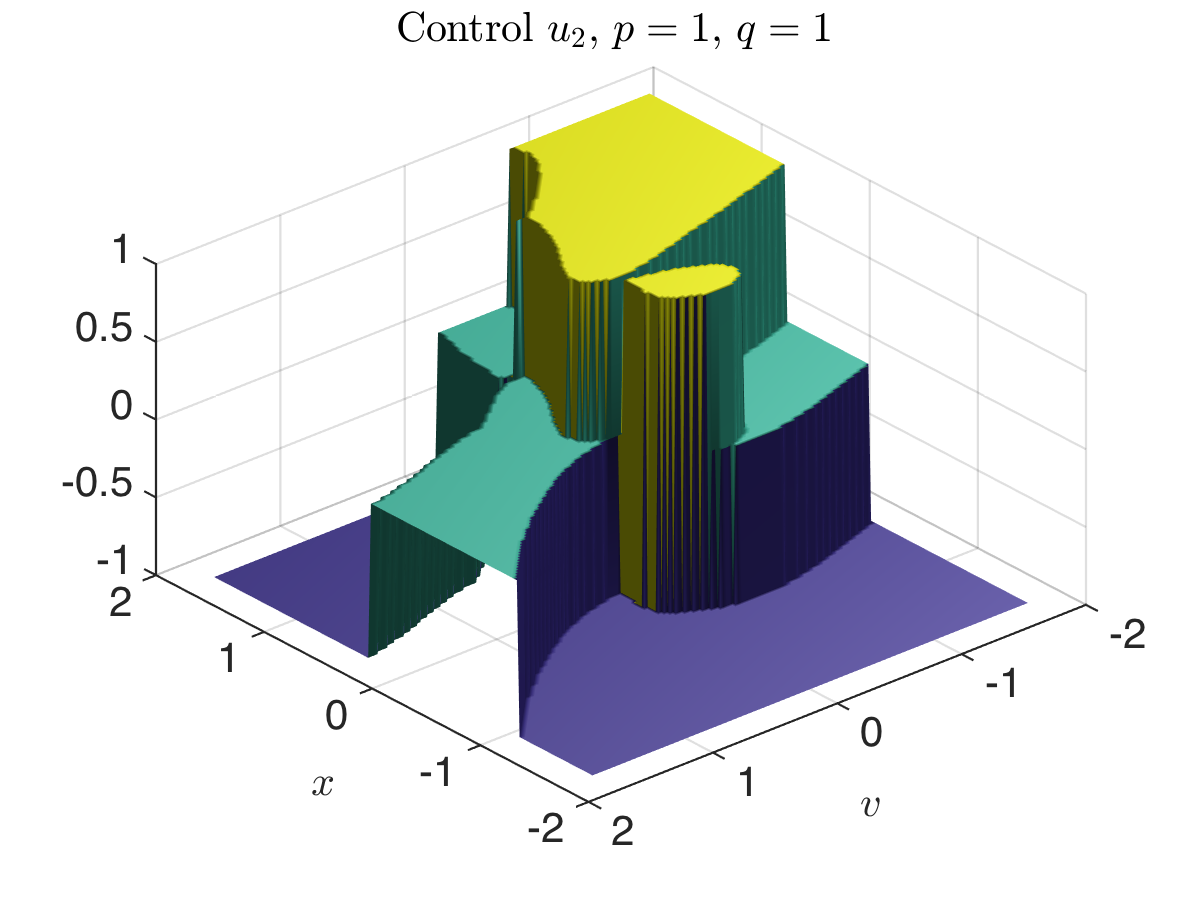}}\hfill
	\subfloat[$u_2 ,p=0.2, q=0.6$]{\label{fig:t3e}\includegraphics[width=0.33\textwidth]{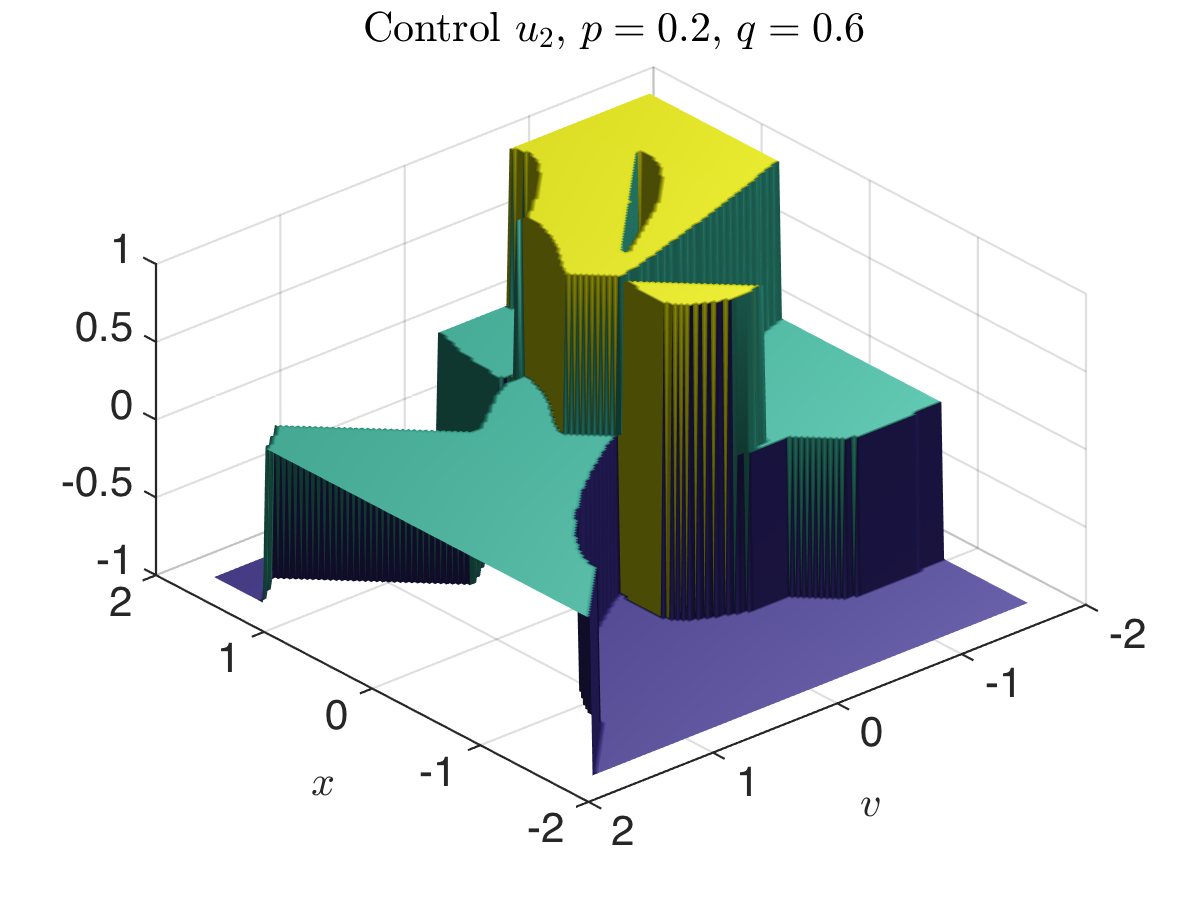}}\hfill
	\subfloat[$u_2 ,p=0.2, q=1$]{\label{fig:t3f}\includegraphics[width=0.33\textwidth]{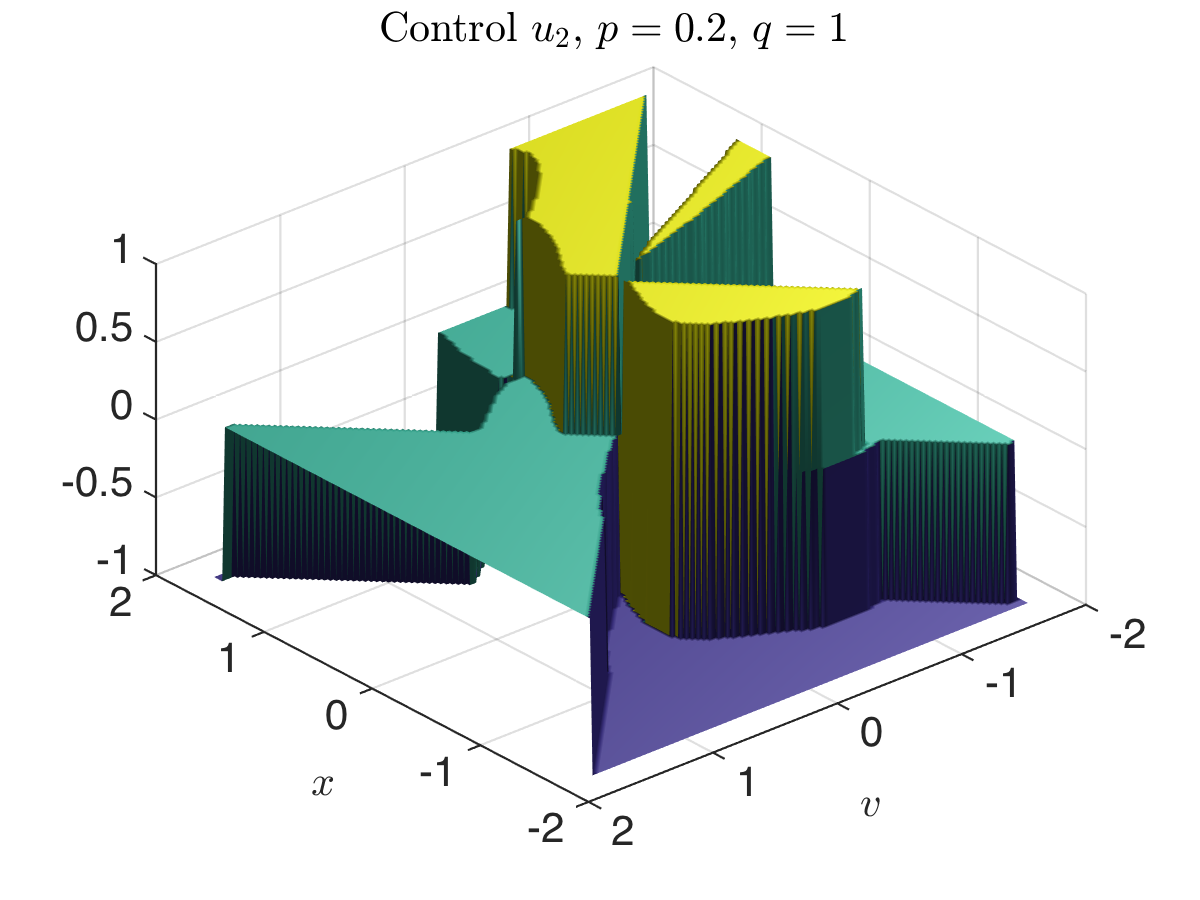}}\\
	\subfloat[$\|u\|_0 ,p=1, q=1$]{\label{fig:t3g}\includegraphics[width=0.33\textwidth]{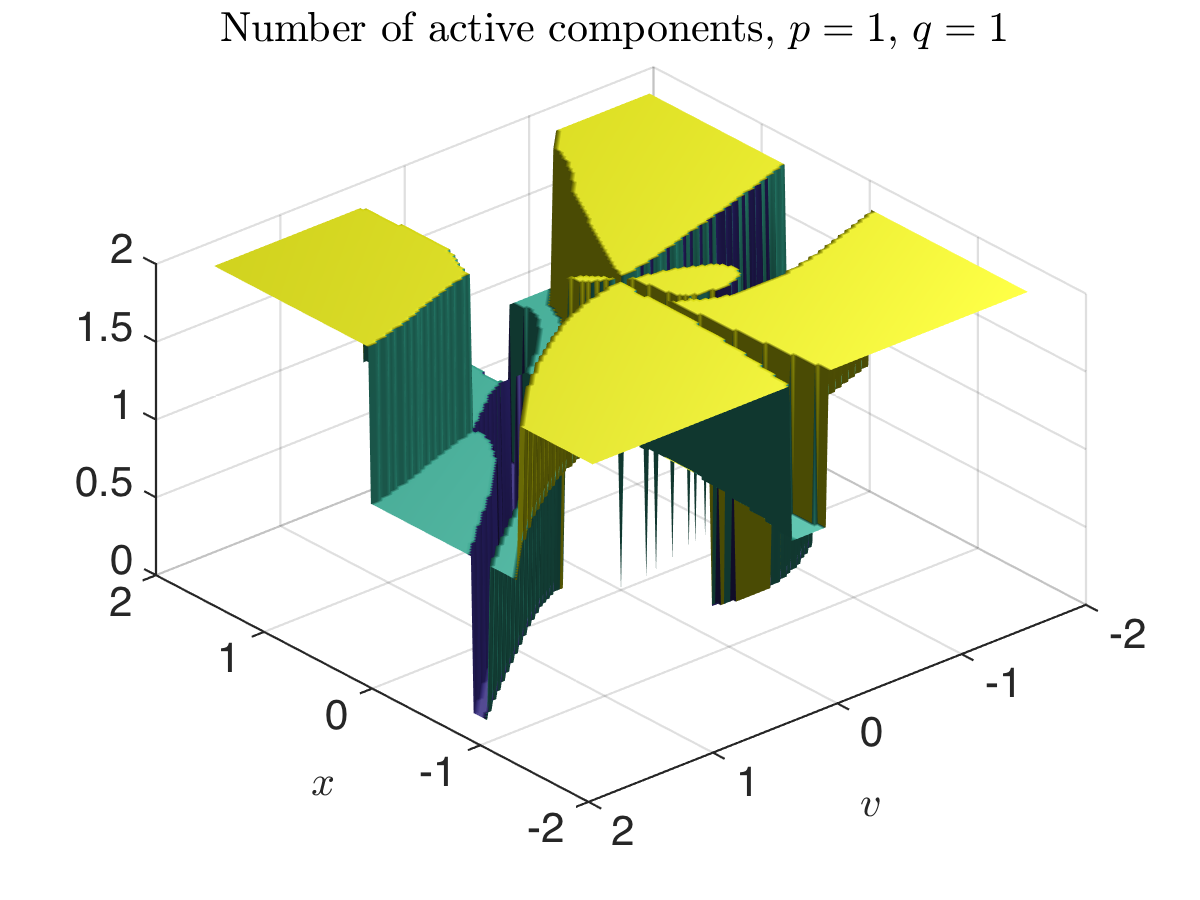}}\hfill
	\subfloat[$\|u\|_0 ,p=0.2, q=0.6$]{\label{fig:t3h}\includegraphics[width=0.33\textwidth]{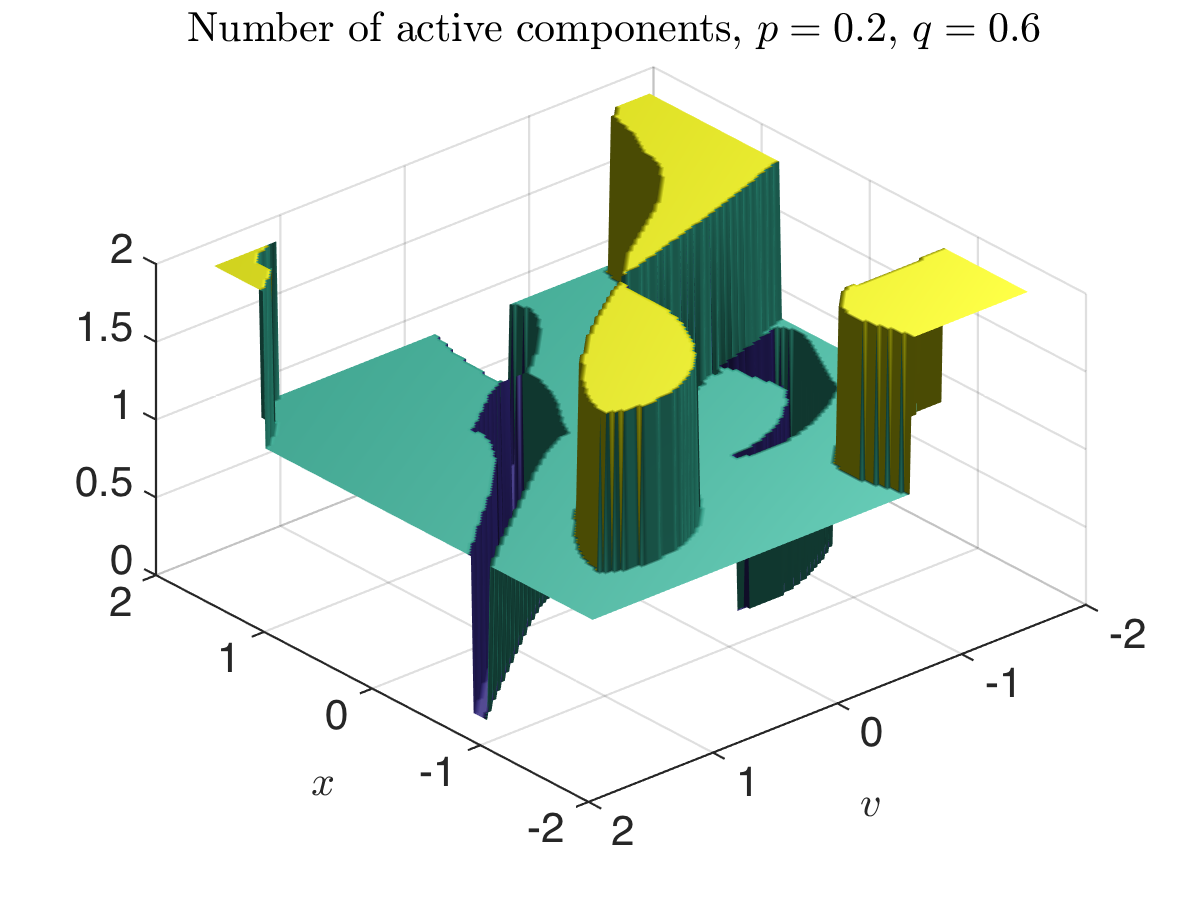}}\hfill
	\subfloat[$\|u\|_0  ,p=0.2, q=1$]{\label{fig:t3i}\includegraphics[width=0.33\textwidth]{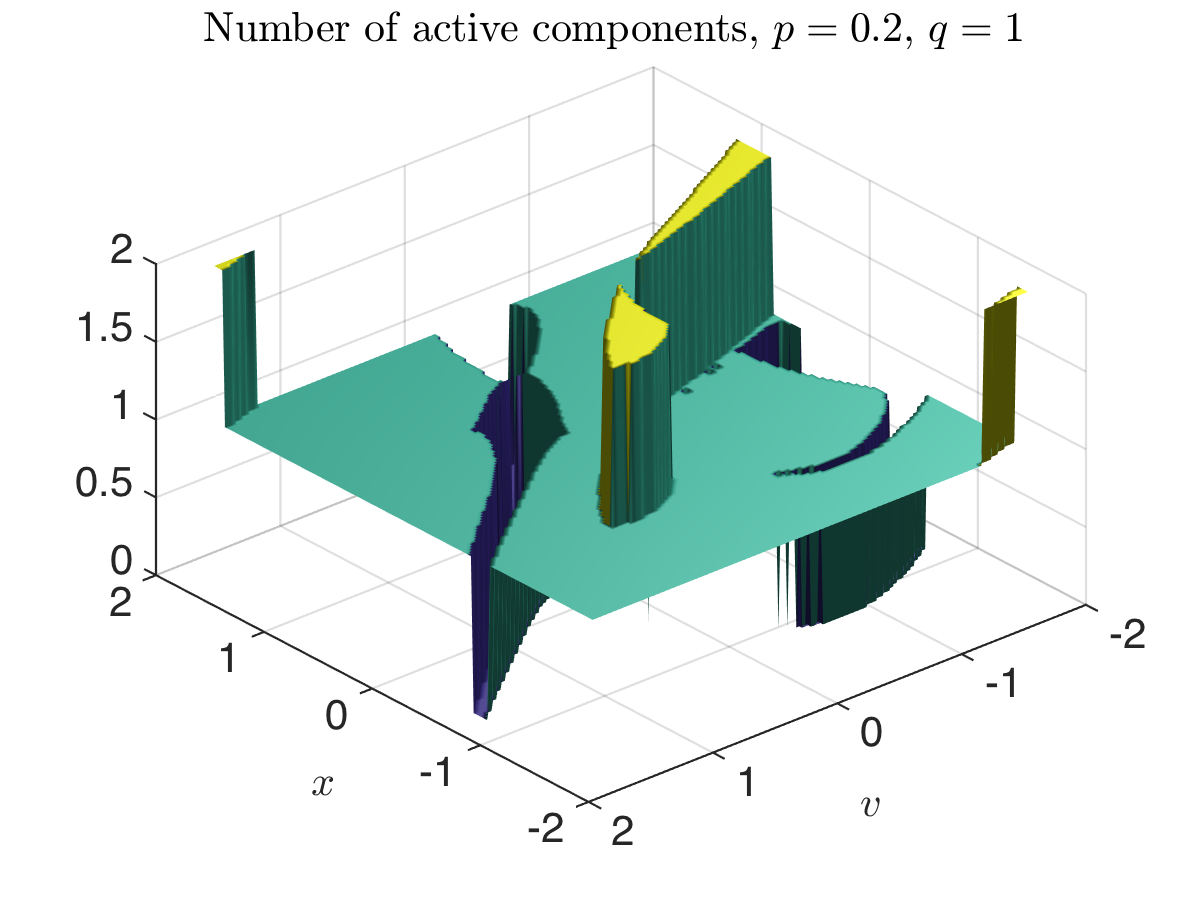}}
	\caption{Optimal controls for the double-well control problem with two bilinear controls. The first two rows show the control variables $u_1$ and $u_2$ for different values of $p$ and $q$.}\label{dw2}
\end{figure}
\noindent We note that:
\begin{itemize}
	\item [a,b,c)] The sparsity region of $u_1$ increases as the ratio $q/p$ increases.
	\item [d,e, f)] The sparsity region of $u_2$ also increases as $q/p$ increases.
	\item [g,h,i)] Overall, the switching pattern of the two control variables becomes dominant as the ratio $q/p$ becomes large. Only a reduced region of the state space requires the simultaneous action of two control variables.
\end{itemize}

\paragraph{Concluding remarks.} In this paper we have studied infinite horizon optimal control problems with a control cost of the form  $\| u \|_p^q$, where $0<p\leq q\leq 1$, leading to a non-convex, non-smooth optimization problem. From the analysis of the associated optimality conditions, we have shown that such control penalizations induce not only sparsity, but also a switching structure in the optimal control field. The switching pattern is determined by the different parameters of the control problem, but most notably, by the value of $q$ and the ratio $q/p$. By means of dynamic programming techniques, we have shown numerically that, for an increased $q/p$ ratio the optimal control has a dominant switching pattern, tending to minimize a counting $\|\cdot\|_0$ measure over an enlarged region of the state space. We believe that an important direction for future research is a thorough study of the interplay between the underlying dynamical structure of the control system and the switching pattern. More concretely, it would be desirable to know whether the sparse/switching control does benefit from the basin of attraction of a given equilibrium point, or whether the inclusion of $\|\cdot\|_p^q$ norms could lead to minimum time-type controllers.

\end{document}